\documentclass[11pt,letterpaper]{amsart}

\usepackage[final
]{showkeys}

\allowdisplaybreaks

\usepackage{AKstyle}
\usepackage[toc,page]{appendix}
\numberwithin{equation}{section}

\usepackage{amsmath}
\usepackage{caption}
\usepackage[labelfont=rm]{subcaption}
\usepackage{scalerel,stackengine}
\stackMath
\newcommand\reallywidehat[1]{%
\savestack{\tmpbox}{\stretchto{%
  \scaleto{%
    \scalerel*[\widthof{\ensuremath{#1}}]{\kern-.6pt\bigwedge\kern-.6pt}%
    {\rule[-\textheight/2]{1ex}{\textheight}}
  }{\textheight}%
}{.8ex}}%
\stackon[1pt]{#1}{\tmpbox}%
}
\usepackage[pagebackref=true, colorlinks]{hyperref}

\newcommand{\startingpower}{n}
\newcommand{\smallerpower}{m}

\newcommand{\subringtwo}{R}
\newcommand{\subringone}{R_1}
\newcommand{\subringthree}{R_1}
\newcommand{\cosetpara}{a}
\newcommand{\subringfactor}{b}

\renewcommand{\phi}{\varphi}

\newtheorem{thmx}{{\normalfont \bf Theorem}}

\hypersetup{pdffitwindow=true,linkcolor=blue,citecolor=blue,urlcolor=blue,menucolor=blue}

\usepackage{comment}

\begin{document}

\author{Jincheng Tang}
\thanks{}
\email{tangent@connect.hku.hk}
\address{Department of Mathematics, The University of Hong Kong}

\author{Xin Zhang}
\thanks{Tang and Zhang are supported by ECS grant 27307320, GRF grant 17317222 and NSFC grant 12001457 from the second author.}
\email{xz27@hku.hk}
\address{Department of Mathematics, The University of Hong Kong}

\title[Sum-product]
{Sum-product in quotients of rings of algebraic integers}

\begin{abstract}
We obtain a bounded generation theorem over $\cO/\fa$, where $\cO$ is the ring of integers of a number field and $\fa$ a general ideal of $\cO$.  This addresses a conjecture of Salehi-Golsefidy. Along the way, we obtain nontrivial bounds for additive character sums over $\cO/\cP^n$ for a prime ideal $\cP$ with the aid of certain sum-product estimates. 
\end{abstract}
\date{\today}
\maketitle

\section{Introduction}

A celebrated theorem of Bourgain, Katz and Tao asserts that for a set $A\subset \mathbb F_p$ of intermediate size, i.e., $p^\delta<|A|<p^{1-\delta}$ for $0<\delta<1$, then $|A+A|+|A\cdot A|> |A|^{1+\epsilon}$ for some $\epsilon>0$ depending only on $\delta$ \cite{BKT04}.  This fundamental theorem has found many diverse applications in mathematics and computer science.  In the same paper, the authors also proved that for a set $A$ of intermediate size in an arbitrary field, if neither $A+A$ nor $A\cdot A$ grows significantly, then $A$ is essentially a very large subset of a scalar multiple of a subfield.  This phenomenon was further generalized by Tao to an arbitrary ring with few zero divisors \cite{Tao2008}.  \par

For rings with many zero divisors such as $\mathbb Z/q\mathbb Z$, $\mathbb F_p[t]/(t^m)$, $\cO/\cP^n$ ($\cO$ is the ring of integers of a number field and $\cP$ is a prime ideal), one needs to make extra assumptions to formulate the expansion statement; otherwise, simple counter examples can be found.  For instance, take $$A=\{p^mk: 1\leq k\leq p^{\frac{m}{10^4}}\}\subset \mathbb Z/p^{2m}\mathbb Z,$$ then $|A+A|<2|A|$ and $|A\cdot A|=1$, but $A$ is not a large subset of a scalar multiple of any subring.  \par

In the paper \cite{BG09} Bourgain and Gamburd proved a sum-product theorem over $\cO/\cP^n$, with which they proved certain spectral gap property, which is also called the super approximation property for Zariski-dense subgroups of $\text{SL}_d(\mathbb Z)$ over prime power moduli.  Later Salehi-Golsefidy proved a sum product theorem over $\cO/\cP^n$ by an entropy method originally due to Lindenstrauss and Varj\'u (see Theorem \ref{1053}), which allowed him to prove super approximation property for more general groups over more general but still restricted moduli \cite{SG19}.  It is worth noting that the expansion constant in \cite{BG09} has uniformity over $n$ but has dependence on $\cP$, while the constant in \cite{SG20} only depends on the extension degree $[\cO: \mathbb Z]$. The uniformity given by Theorem \ref{1053} will be a crucial ingredient in the proof of our main theorem.  \par

Closely related to sum-product expansion is the bounded generation phenomenon, which in the setting of $\mathbb F_p$, says that for a set $A\subset \mathbb F_p$ with $|A|>p^\delta$, then for some $r_1, r_2\in\mathbb Z$ depending only on $\delta$, we have $\sum_{r_2}A^{r_1}=\mathbb F_p$.  Here $A^{r_1}$ is the $r_1$-fold product of the set $A$, and the set $\sum_{r_2}A^{r_1}$ is the $r_2$-fold sum of the set $A^{r_1}$.  Bounded generation follows easily from sum-product expansion.  To see this, let $\epsilon$ (depending on $\delta$) be the expansion constant.  Then by taking either sum or product of $A$, one can obtain a larger set with size increase by at least a factor of $p^{\delta\epsilon}$.  Iterate this for at most $[\frac{1}{\delta\epsilon}]$ times, we obtain a sum-product set of $A$ of size at least, say $p^{\frac{9}{10}}$.  Then a standard Fourier argument shows that a further sum-product set is $\mathbb F_p$.  \par
In practice, the logic is in reverse order. Proving a bounded generation result is often on the way of proving a sum-product expansion theorem. See \cite{BKT04}, and see also Bourgain's treatment of $\mathbb Z_q$ for an arbitrary $q$ \cite{Bou08}.


In the paper \cite{SG20} Salehi-Golsefidy made the following bounded generation conjecture on a general quotient of the integer ring of a general number field: 
\begin{conj}[Conjecture 6, \cite{SG20}]\label{BGC} Suppose $0<\delta\ll 1$, $d$ is a positive integer, and $N_0\gg_{d,\delta} 1$ is a positive integer.  Then there are $0<\epsilon: =\epsilon(\delta, d)$ and positive integers $C_1=C_1(\delta, d), C_2=C_2(\delta, d), C_3=C_3(\delta, d)$ such that for any number field $K$ of degree at most $d$ the following holds:  Let $\mathcal O$ be the ring of integers of $K$.  Suppose $\mathfrak a$ is an ideal of $\mathcal O$ such that $N(\mathfrak a):= |\mathcal O/\mathfrak a |\geq N_0$, and suppose $A\subseteq \mathcal O$ such that 
$$|\pi_{\mathfrak a}(A)|\geq|\pi_{\mathfrak a}(\mathcal O)|^{\delta}.$$
Then there are an ideal $\mathfrak a'$ of $\mathcal O$, and $a\in\mathcal O$ such that 
\begin{align*}
&\mathfrak a^{C_1}\subseteq \mathfrak a',\\
&\pi_{\mathfrak a'}(\mathbb Za)\subset \pi_{\mathfrak a'}\left(\sum_{C_3}A^{C_2} -\sum_{C_3}A^{C_2}\right),\\
& |\pi_{\mathfrak a'}(\mathbb Za)|\geq N(\mathfrak a)^\epsilon.
\end{align*}
\end{conj}
\noindent Here, $\pi_\fa \text{ (resp. } \pi_{\fa'})$ is the reduction map $\cO\rightarrow \cO/\fa \text{ (resp. }\cO\rightarrow \cO/\fa' )$, the set $A^{C_2}=\{x_1\cdots x_{C_2}: x_i\in A, 1\leq i\leq C_2 \}$ is the $C_2$-fold product of the set $A$, and the set $\sum_{C_3}A^{C_2}=\{x_1+\cdots+x_{C_3}: x_i\in A^{C_2}, 1\leq i\leq C_3\}$ is the $C_3$-fold sum of the set $A^{C_2}$. Roughly speaking, Conjecture \ref{BGC} is a quantitative version of the statement that given a set $A\subset \mathcal O$ and an ideal $\fa\subset \mathcal O$ such that if $|\pi_{\fa}(A)|$ is not too small, then modulo an ideal $\fa'$ comparable to $\fa$, a sum-product set of $A$ contains a thick arithmetic progression.  \par
Salehi-Golsefidy made Conjecture \ref{BGC} for the purpose of proving the following super approximation conjecture in full generality \cite{BV12}: 

\begin{conj}[Question 2, \cite{GV12}] \label{SAC}Let $\Gamma<\text{SL}_n(\mathbb Z)$ be finitely generated, then $\Gamma$ has the super approximation property with respect to all positive integers if and only if the identity component $\mathbb G_0$ of the Zariski closure $\mathbb G$ of $\Gamma$ is perfect, i.e. $[\mathbb G_0, \mathbb G_0]=\mathbb G_0$.
\end{conj} 

Although Conjecture \ref{SAC} is formulated in $\mathbb Z$, algebraic integers are involved because in many previous works on Conjecture \ref{SAC}, one needs to diagonalize matrices at certain stage, which in general is only possible in an extended field.

Our main theorem in this paper is 
\begin{theorem}\label{mainthm} Conjecture \ref{BGC} is true. 
\end{theorem}

With Theorem \ref{mainthm} at hand, we manage to prove a new case of Conjecture \ref{SAC} in a separate paper \cite{TZ23b}: 

\begin{theorem}\label{main}
Let $S\subset  {\normalfont \text{SL}}_2(\mathbb Z)\times {\normalfont \text{SL}}_2(\mathbb Z)$ be a finite symmetric set, and assume that it generates a group $G$ which is Zariski-dense in $ {\normalfont \text{SL}}_2\times  {\normalfont \text{SL}}_2$, then $G$ has the super approximation property with respect to all positive integers. \end{theorem}

In the paper \cite{Bou08} Bourgain provided a scheme of glueing local pieces of $\mathbb Z_q$.  In the proof of Theorem \ref{mainthm}, we adopt the scheme of Bourgain, but we also need to make significant changes at several places.  To prepare explaining some core ideas in the proof, we mention another closely related phenomenon to sum-product expansion, which is the Fourier decay.  \par

Using sum-product theorem in $\mathbb F_p$, Bourgain, Glibichuk and Konyagin proved that for any set $A\subset \mathbb F_p$ with $|A|>p^\delta, 0<\delta\leq 1$, there exists $k\in\mathbb Z_+$ and $\epsilon>0$ depending only on $\delta$ such that if we let $\mu$ be the uniform probability measure supported on $A$, then $\hat{\mu^{\otimes k}}(\xi)< p^{1-\epsilon}$ for any nonzero $\xi\in \hat{\mathbb F_p}$, where ${\mu^{\otimes k}}$ is the $k$-fold multiplicative convolution of $\mu$ \cite{BK06}.  This is one instance of Bourgain's brilliance of ``connecting analysis to combinatorics'', and the same strategy has been applied to Bourgain's several other papers, for instance \cite{BC06}, \cite{Bou08}. \par
 To see how in reverse order Fourier decay implies sum-product expansion (which turns out to be important for collecting some local information in the proof of Theorem \ref{mainthm}),  we take $\mu^{(\otimes k)\oplus r}$ the $r$-fold additive convolution of $\mu^{\otimes k}$ for sufficiently large $r$, then every nonzero Fourier coefficient of $\mu^{(\otimes k)\oplus r}$ is very small. An application of Plancherel theorem implies that the support of $\mu^{(\otimes k)\oplus r}$, which is exactly a sum-product set of $A$, must be very large.  This will eventually {imply} bounded generation and sum-product. \par

Now we give a sketch and some core ideas of the proof of Theorem \ref{mainthm}.  Let $A$ be given as in Conjecture \ref{BGC}.  Up to passing to a related set, we can assume $A$ consists of invertible elements, and  $A$ is not concentrated in any unital subring of $\langle A\rangle$ (the ring generated by $A$) and its affine translates.  A general subring of $\cO/\fa$ can be very complicated, but a theorem of Salehi-Golsefidy (Theorem \ref{structure}) says that for $n$ sufficiently large for any unital subring $R$ of $\cO/\cP^n$, there is an integer $m>cn$ for some $0<c<1$ independent of $\cP$ such that $R(\mod \cP^m)$ is very close to $\cO_0(\mod \cP^m)$ for $\cO_0$ the ring of integers of a subfield of $K$.  For the purpose of proving Theorem \ref{main}, it suffices to pass the modulus $\cP^n$ to $\cP^m$.  Goursat's lemma then allows us to treat a general ideal $\fa$.  \par

Motivated by Theorem \ref{structure}, we define a class of subrings in $\cO/\cP^n$ of a simple form, and we call these rings \emph{congruential} rings. For $m\leq n$, let $\pi_{n,m}:\mathcal O/\mathcal P^n\rightarrow\mathcal O/\mathcal P^m$ be the canonical projection. We abbreviate $\pi_{n,m}$ to $\pi_m$ if the initial ring $\cO/\cP^n$ is clear in the context.  We say a unital ring $R<\cO/\cP^n$ is {congruential} if there exists a number field $K'<K$ with ring of integers $\cO'$, and a positive integer $m\leq n$ such that $R=R_{\cO',m}$, where
\begin{align}\label{1329}
R_{\cO',m}=\pi_{n,m}^{-1}(\cO'(\mod \cP^m)).
\end{align}

On the first stage of the proof of Theorem \ref{mainthm}, we need to develop some character sum estimate analogous to the result of Bourgain, Glibichuk and Konyagin in the prime field case.  To explain our results, we first introduce some notations.  As an additive group, $\cO/\cP^n$ has a dual group $\hat{\cO/\cP^n}$ consisting of all additive characters of $\cO/\cP^n$. The dual group $\hat{\cO/\cP^n}$ is naturally a $\cO/\cP^n$-module: For $y\in \cO/\cP^n$ and $\chi\in \hat{\cO/\cP^n}$, let $y\chi$ be the character such that $y\chi(x)=\chi(xy)$.  An additive character $\chi\in\hat{\mathcal O/\mathcal P^n}$ is \emph{primitive} if $\chi$ is not induced by a character in $\mathcal O/\mathcal P^m$ by the pullback of $\pi_{n,m}$.  \par \par
Let $\mu$ be a probability measure on $\mathcal O/\mathcal P^n$, the Fourier transform $\hat{\mu}(\chi)$ is defined as 
$$\hat{\mu}(\chi)=\sum_{x\in\mathcal O/\mathcal P^n}\mu(x)\chi(x).$$

For two measures $\mu, \nu$ on $\mathcal O/\mathcal P^n$, the 
additive convolution is
$$\mu*\nu(x)=\sum_{y\in \mathcal O/\mathcal P^n}\sum_{z\in\mathcal O/\mathcal P^n}\mu(y)\nu(z)\boldsymbol{1}\{y+z=x\},$$
and the multiplicative convolution is
$$\mu\otimes\nu(x)=\sum_{y\in \mathcal O/\mathcal P^n}\sum_{z\in\mathcal O/\mathcal P^n}\mu(y)\nu(z)\boldsymbol{1}\{yz=x\}.$$

\par

We also denote by $R^*$ the multiplicative group of invertible elements in $R$, and $\hat{R}$ the additive dual of $R$. \par

 We have 

\begin{thm}\label{exp1} Suppose $d$ is a positive integer.  For all $\gamma>0$ there exist $\epsilon=\epsilon(\gamma, d)>0, \tau=\tau(\gamma, d)>0$, $k=k(\gamma, d)\in\mathbb Z_+$ and $N=N(\gamma, d)$ such that the following holds:  Let $K$ be a number field with field extension degree over $\mathbb Q$ at most $d$, $\mathcal O$ be the ring of integers of $K$ and $\mathcal P$ a prime ideal of $\mathcal O$. Assume $|\mathcal O/\mathcal P^n|>N$.  Let $\mu_i, 1\leq i\leq k$ be probability measures on $\mathcal O/\mathcal P^n$. Assume all $\mu_i$ satisfy the following non-concentration condition: \par
\noindent For all congruential subrings $R<\cO/\cP^n, a, b\in \cO/\cP^n$, with 
$$[\cO/\cP^n: bR]>|\cO/\cP^n|^\epsilon,$$
we have 
\begin{align}\label{nonconcentration1}
\mu_i(a+bR )<[\cO/\cP^n:  bR]^{-\gamma},
\end{align}
where $[\cO/\cP^n:  bR]$ is the index of $bR$ in $\cO/\cP^n$ as an abelian subgroup. 
Then, for any primitive additive character $\chi$ on $\mathcal O/\mathcal P^n$, we have 
\begin{align}
|\reallywidehat{\mu_1\otimes\mu_2\otimes\cdots \otimes\mu_k}(\chi)| <|\cO/\cP^n|^{-\tau}.
\end{align}
\end{thm}

%
%

Theorem \ref{exp1} is a generalization of Theorem 2 in \cite{Bou08}, where Bourgain proved the case $\mathcal O=\mathbb Z$ and each $\mu_i$ is the counting measure of a set $A_i\subset \mathbb Z/(p^n)$. Write $\cO/\cP=\mathbb F_{p^l}$.  If $l=1$, the analysis is the same as the case $\mathbb Z/(p^n)$.  If $l>1$, a main challenge is to deal with subrings of $\cO/\cP^n$.  In the proof of Theorem \ref{exp1}, we need to develop a more measure theoretical friendly sum-product theorem over $\cO/\cP^n$ (Theorem \ref{sumproduct}).  Once Theorem \ref{sumproduct} is established, we then follow the procedure in \cite{Bou08} to pass Theorem \ref{sumproduct} to Theorem \ref{exp1}.  

We will then apply Theorem \ref{exp1} to a glueing process similar to Bourgain's treatment for $\mathbb Z/(q)$ (Part II, (3) of \cite{Bou08}) to conclude Theorem \ref{mainthm}.  More specifically, write $\fa=\prod_{i=1}^{s}\cP_i^{n_i}<\cO$ in prime ideal decomposition, and let $\mu_A$ be the probability counting measure of $A$ so $\mu_A$ satisfies the non-concentration property \eqref{nonconcentration1} with $\cP^n$ replaced by $\alpha$ and $\gamma$ replaced by a number depending on $\delta$.  This non-concentration property is preserved under additive and multiplicative convolution. Take $k$ the implied constant in Theorem \ref{exp1} and $r$ sufficiently large. Then $\mu=(\mu_A^{\otimes k})^{\oplus r}$ has significant Fourier decay.  We want to construct an ideal $\fa'\subset \fa$, rearranging indices, let's say, $\fa'=\prod_{i=1}^{s'}\cP_i^{n_i}$ for $s'\leq s$, such that $\mu$ enjoys a finer local non-concentration property with respect to the filtration $\{\cO/\prod_{i=1}^{j}\cP_i^{n_i}\}_{j=1}^{s'}$, i.e., given any time $1\leq j<{s'}$, for any $\xi\in \cO$ (up to an exceptional set $C_j$), the conditional probability measure $\mu\vert_{\{x\in\cO/\fa: x\equiv \xi(\mod\cO/\prod_{i=1}^{j}\cP_i^{n_i} )\}}$ (capturing elements who share the same history with $\xi$ up to time $j$), satisfies essentially the same non-concentration property at the next local place $\cO/\cP_{j+1}^{n_{j+1}}$ after a pushforward.  \par

At this moment it is worth mentioning that \eqref{nonconcentration1} in Theorem \ref{exp1} is a delicate assumption to find.  We need to be careful not to assume too little to draw any conclusion or too much (e.g., replace ``all congruential subrings" by ``all subrings''), { so that in the application to the filtration process we do not have too many conditions to check, and as a result we can push $s'$ large so that $|\cO/\fa'|>|\cO/\fa|^{\delta_0}$ for some $\delta_0>0$ depending only on $\delta$}. \footnote{In Bourgain's treatment of $\mathbb Z/q\mathbb Z$ \cite{Bou08}, he assumes $|A+A|$ is not much larger than $|A|$, from which he can push $s'$ to the extend that $|\cO/\fa'|>|\cO/\fa|^{1-\delta_0}$ for some small $\delta_0$.  We do not make this assumption. It also seems impossible to find a uniform exceptional set $C=C_j$ for all $j$ as claimed in \cite{Bou08}, but it didn't affect the analysis.} Then if we let $(\mathbb E_j)_{1\leq j\leq s'}$ be the conditional expectation operator associated to this filtration, we can iteratively apply Theorem \ref{exp1} to show that the entropy of $\mathbb E_1(\mu)$, as well as the relative entropies of $\mathbb E_{j+1}(\mu)$ with respect to $\mathbb E_{j}(\mu)$, $1\leq j<s'-1$, are close to those of the uniform distribution measure.  This implies the support of $\mathbb E_{s'}(\mu)$ must be a very large set in $\cO/\fa'$, a further sum-product of which gives Theorem \ref{mainthm}.  \par
 
The techniques in this paper are in great debt to \cite{Bou08}.  We also use two theorems from \cite{SG20} in an essential way. A contribution from this paper is that we can combine Bourgain's analysis in \cite{Bou08} and analysis of subrings of $\cO/\fa$, which takes place throughout this paper and was previously a well known difficulty.

\noindent {\bf Acknowledgements} \par
We thank Alex Kontorovich for discussions during the production of this work, and Zeev Rudnick, Nicolas de Saxc\'e, He Weikun and Tran Chieu Minh for {numerous} comments/suggestions on previous versions of this paper.

\newpage
\section{Sum-product Preliminaries}
The goal of this section is to prove Theorem \ref{sumproduct} below, which is a crucial ingredient in the proof of Theorem \ref{exp1} in the following sections.

\begin{theorem}[Sum-product Theorem]\label{sumproduct}
Suppose $d$ is a positive integer.  Given $\delta_1,\delta_2>0$, there exists $\epsilon=\epsilon(\delta_1, \delta_2)$, and $\delta_3=\delta(\delta_1, \delta_2, \epsilon, d)>0$ such that the following holds { for sufficiently large $n$}:  { Let $K$ be a number field with $[K: \mathbb Q]\leq d$}.  Let $\mathcal O$ be the ring of integers of $K$ and $\mathcal P$ is a prime ideal of $\cO$. Write $\cO/\cP=\mathbb F_{p^{d_0}}$ for some $d_0\leq d$, and $q=|\cO/\cP^n|$. Suppose $A\subset \mathcal O/\mathcal P^n$ such that
\begin{align}\label{1020}
q^{\delta_1}<|A|<q^{1-\delta_1},
\end{align}
and for all congruential subrings $R_1<\cO/\cP^n, a, b\in \cO/\cP^n$, with 
$$[\cO/\cP^n: bR_1]>q^\epsilon,$$
we have
\begin{align}\label{2226}
|A\cap (a+b\subringone)|<[\mathcal O/\mathcal P^n: b\subringone]^{-\delta_2}|A|.
\end{align}
Then 
\begin{align}
\label{expansion}
|A+A|+|A\cdot A|>|A|^{1+\delta_3}.
\end{align}
\end{theorem}

We need the following theorem { due} to Salehi-Golsefidy, which relates unital subrings of $\mathcal O/\cP^n$ to subfields of $K$ when $n$ is large.

\begin{thmx}[Salehi-Golsefidy, \cite{SG20}, Theorem 7]\label{structure}
\textit{Let $\mathcal O$ be the ring of integers of a finite extension $K$ of $\mathbb Q$.  Let $\mathcal P$ be a prime ideal of $\cO$, and $R$ be a unital subring of $\mathcal O/\mathcal P^n$.  Let $C$ be an integer which is at least 3.  Suppose $n$ is an integer and }
\begin{equation}\label{lowerbd}
n\gg_C[K:\mathbb Q].
\end{equation}
\textit{
  Then there are integers $a$ and $b$, and a subfield $K_0$ of $K$ such that 
\begin{align}
&b-a\gg_{C,[K:\mathbb Q]}n, \text{ and } n\geq b\geq Ca,\\
&\pi_{\cP^b}(\cO_0\cap \cP^a)=\pi_{\cP^b}(R\cap \cP^a), \text{ and }\pi_{\cP^b}(R)\subseteq\pi_{\cP^b}(\cO_0) \label{1116}
\end{align}
where $\cO_0$ is the ring of integers of $K_0$.}
\end{thmx}

\begin{remark}
Theorem \ref{structure} was stated in terms of the non-Archimidean completions $\mathbb Q_p$ and $\cO_\cP$ as Theorem 7 in \cite{SG20}, which directly implies Theorem \ref{structure} stated here by an application of Krasner's Lemma.  
\end{remark}

{

We explain how Theorem \ref{structure} is used in our paper.  Let $R$ be as given in Theorem \ref{structure}. For $i\leq [ \frac{n}{e}-1]$, we have a natural injection $$(p^i)\cap R (\mod (p^{i+1}) )\rightarrow \cO/(p): p^ix\mapsto  x(\mod p),$$
Define $D_i\subset \cO/(p)$, the $i$-th level group of $R_1$ to be the image of this map. We then have an ascending chain of $p$-groups: $D_0\subset D_1\subset \cdots \subset D_{[\frac{n}{e}]-1}$. \par

We first look at the situation when $D_i$ are constant. 
\begin{corollary}\label{324}
Let $R$ be an unital ring in $\cO/\cP^n$. There exists an absolute constant $M$ depending only on $[K:\mathbb Q]$ such that if $\frac{n}{e}\geq M$ and $D_0=D_1=\cdots =D_{[\frac{n}{e}]-1}$, then there is a ring of integers $\cO_0<\cO$ such that $R(\mod \mathcal P^n)=\cO_0(\mod \mathcal P^n)$.
\end{corollary}
\begin{proof}
Take $C=3[K:\mathbb Q]$ and let $M$ be the implied lower bound for $n$ at \eqref{lowerbd} in Theorem \ref{structure}, which only depends on $[K:\mathbb Q]$.  Suppose $n\geq eM$ and $D_0=D_1=\cdots=D_{[\frac{n}{e}]-1}$, then Theorem \ref{structure} gives $a, b\in \mathbb Z^+, b\geq 2a$ and an integral ring $\mathcal O_0$, such that 
\begin{align}
\label{253}
p^a R(\mod p^b)= p^a (\mathcal O_0\mod p^b),\end{align} 
which directly implies 
\begin{align}
R(\mod p^{b-a})= \mathcal O_0(\mod p^{b-a}).
\end{align}
We show one more step that $R(\mod p^{b-a+1})= \mathcal O_0(\mod p^{b-a+1})$. Corollary \ref{324} then follows from iterating the same argument. \par
First by \eqref{253}, $$p^{b-a}\cO_0(\mod (p^{b-a+1}))= (p^{b-a})\cap R(\mod (p^{b-a+1})).$$ Take an arbitrary nonzero element $p^{b-a}y_0$ in $p^{b-a}\cO_0(\mod p^{b-a+1})$. If $R $ contains an element of the form $x+p^{b-a}y$, $x(\mod p^{b-a+1})\in \cO_0^*(\mod p^{b-a+1})$, $y(\mod p^{b-a+1})\not\in \cO_0(\mod p^{b-a+1})$. Then $$w:=(x+p^{b-a}y)^{-1}(x+p^{b-a}(y+y_0)= 1+x^{-1}y_0 p^{b-a}(\mod p^{b-a+1}), $$ and 
$w^2-w=x^{-1}y_0p^{b-a} \not\in \mathcal O_0(\mod p^{b-a+1})$, contradicting \eqref{253}. So all invertible elements in $R(\mod p^{b-a+1})$ lie in $\cO_0(\mod p^{b-a+1})$. This forces all elements of $R(\mod p^{b-a+1})$ lie in $\cO_0(\mod p^{b-a+1})$ as well, and thus \linebreak $R(\mod p^{b-a+1})= \cO_0(\mod p^{b-a+1})$ by comparing cardinality. 

\end{proof}
In a general situation, since $D_i(\mod \cP)$ are ascending $p$-groups, the number of changes of $D_i$ is controlled by the extension degree $[K: \mathbb Q]$. Therefore, for arbitrarily large $C>0$, when $n$ is sufficiently large (depending only on $[K:\mathbb Q]$ and $C$), we can find $m \gg_{[K:\mathbb Q]} n$ such that $D_i$ is constant when $i\in [\lceil \frac{m}{C} \rceil, m]$. Theorem \ref{structure} implies there is $\cO_0<\cO$, such that $R(\mod \cP^m) < \cO_0(\mod \cP^m)$ and $(p^{\lceil \frac{m}{C}\rceil })\cap R =(p^{\lceil \frac{m}{C}\rceil })\cap \cO_0 (\mod p^m) $.

 }

We also record the following two sum-product theorems that we are going to apply. 

\begin{thmx}[Bourgain-Katz-Tao, \cite{BKT04}, Theorem 4.3]\label{BKT04}\textit{ Let $A$ be a subset of a finite field $F$ such that $|A|>|F|^{\delta}$ for some $0<\delta<1$, and suppose that $|A+A|, |A\cdot A|\leq K|A|$ for some $K\gg 1$.  Then there exists a subfield $G$ of $F$ of cardinality ${ |G|}\leq K^{C(\delta)}|A|$, a non-zero field element $\xi\in F-\{0\}$, and a set $X\subseteq F$ of cardinality $|X|\leq K^{C(\delta)}$ such that $A\subseteq \xi G\cup X$. }

\end{thmx}

\begin{thmx}[Salehi-Golsefidy, \cite{SG20}, Theorem 9]\label{1053} \textit{Let $\delta>0$, and $d$ is a positive integer.  There exist $\epsilon=\epsilon(\delta, d), \tau=\tau(\delta, d)$ such that the following holds:  Let $K$ be a number field with $[K: \mathbb Q]<d$ and $\cO$ be the ring of integers of $K$.  Let $\cP$ be a prime ideal of $\cO$.  Suppose {$n$} is sufficiently large depending on $\delta$.  Let $A\subset \cO/\cP^n$ with the following property: 
\begin{align}
\nonumber& |A|<|\cO/\cP^n|^{1-\delta},\\
&\label{618}\pi_{\cP^2}( A)=\cO/\cP^2,\\
\nonumber&|\pi_{\cP^m}(A)|>|\cO/\cP^m|^{\delta} \text{ for any }m\geq n\epsilon.
\end{align} 
Then $$|\sum_{12}A^7-\sum_{12}A^7|>|A|^{1+\tau}.$$ }
\end{thmx}

\begin{remark} In \cite{SG20}, Salehi-Golsefidy proved Theorem \ref{1053} with an entropy approach due to Lindenstrauss and Varj\'u for the special case that $|\cO/\cP|$ is sufficiently large.  The critical step is a regularization process, which resulted in a rooted tree, where the vertices in the $i$th level of the tree are elements of $\pi_{\cP^{i}}(\mathcal O)$.  If $|\cO/\cP|$ is not large but $|\cO/\cP^n|$ is large, then in the regularization process, one can take a sufficiently large integer $M=M(\delta)$, and let the vertices in $i$th level be the elements of $\pi_{\cP_i^{Mi}}(\mathcal O)$. Then one can deal with this case by the method in \cite{SG20} with minor modification.   See also Bourgain's treatment of $\mathbb Z/(p^n)$ for $p$ not large but $n$ large (Section 6, Part I of \cite{Bou08}).  \par
\end{remark}

\begin{remark} In \cite{BG09} Bourgain-Gamburd obtained a sum-product theorem in $\mathcal O/\cP^n$ using Fourier method with constants dependent on $\cP$ but independent of $n$. The dependence on $ \cP$ can be removed if one inputs Theorem \ref{structure} in the argument.  
\end{remark}

{\begin{remark} Both Theorem \ref{sumproduct} and Theorem \ref{1053} are sum product theorems over $\cO/\cP^n$. The reason we work Theorem \ref{sumproduct} out of Theorem \ref{1053} is that compared to Condition \eqref{618} of Theorem \ref{1053}, the non-concentration assumption \eqref{2226} in Theorem \ref{sumproduct} is stable under minor changes of the set $A$, which is more friendly to our Fourier approach performed in the moduli gluing process. Still, the assumption \eqref{2226} guarantees $A$ generates a subring of $\cO/\cP^n$ of almost full size. By comparison, in the statement of Theorem \ref{1053}, one can make $A$ fail assumption \eqref{618} by removing $|\cO/\cP^2|$ many elements from $A$, which is a small set if $n$ is large.

\end{remark}
}

\subsection{Reduction to invertible sets}
To prove Theorem \ref{sumproduct}, it suffices to assume $A$ consists of invertible elements.  Indeed, given $\delta_1, \delta_2$, let {$\epsilon'=\epsilon(\delta_1/2,\delta_2/2), \delta'=\delta_3(\delta_1/2, \delta_2/2)$} be the implied constants in this special case.  We claim that we can take 
\begin{align}\label{550}
{\epsilon(\delta_1, \delta_2)}=\min\{\frac{\epsilon'}{2}, \frac{\delta_1\delta'}{2} \}
\end{align}
and
$$\delta_3(\delta_1, \delta_2)=\frac{\delta'}{2},$$
 for which Theorem \ref{sumproduct} holds for a general set $A$.

Indeed, let $\fp$ be a uniformizing element in $\cO/\cP^n$ and for a general set $A\subset \cO/\cP^n$ satisfying \eqref{1020} and \eqref{2226}, write

$$A=A_0\sqcup A_1\sqcup \cdots \sqcup A_{n-1}\sqcup A_n ,$$
where $$A_i=\{x\in A: \fp^{i}||x \}$$ for $0\leq i\leq n-1$
and 
\begin{equation}
A_n=\begin{cases}\{0\}&\text{ if } 0\in A\\ \emptyset &\text{ if } 0\not\in A\end{cases}. 
\end{equation}
Here $\fp^i|| x$ means $x$ can be written as $x=\fp^i y$ for some invertible element $y\in \cO/\cP^n$. 

Take $i$ such that $|A_i|\geq \frac{|A|}{n}$.  Suppose $i> n\epsilon$, then from \eqref{2226},
$$\frac{|A|}{n}\leq |A_i|< |A|[\cO/\cP^n: \cP^i /\cP^n]^{-\delta_2}= |A||\cO/\cP^i|^{-\delta_2},$$
which implies
$$q^{\epsilon \delta_2}<n.$$
This is a contradition if $q$ is sufficiently large.  Thus 
\begin{align}\label{551}
i\leq n\epsilon.
\end{align}  \par

If we consider the set $A_i':= A_i/\mathfrak p^i\subset \mathcal \cO/\mathcal P^{n-i}$, then $A_i'\subset (\cO/\mathcal P^{n-i})^*$. {We show $A_i'$ essentially satisfies the same non-concentration property as $A$. First, it follows from \eqref{550} and \eqref{551}} that
$$|\cO/\cP^{n-i}|^{{\frac{\delta_1}{2}}}<|A_i'|< |\cO/\cP^{n-i}|^{1-{\frac{\delta_1}{2}}}.$$
For any congruential subrings $R< \cO/\mathcal P^{n-i}, a,b\in  \cO/\mathcal P^{n-i}$ with $$[\cO/\cP^{n-i}: bR]>|\cO/\cP^{n-i}|^{\epsilon'},$$

we have 
$$[\cO/\cP^{n}: \fp^i bR]=[\cO/\cP^{n-i}: bR]>q^{\frac{\epsilon'}{2}}>q^{\epsilon},$$
{ where we have viewed $\mathfrak p^i bR$ as an abelian subgroup of $\cO/\cP^n$.}

Therefore,  
\begin{align*}
&|A_i'\cap (a+bR)|=|\mathfrak p^i A_i'\cap (\mathfrak p^i a+\mathfrak p^ibR)| \\
\leq& |A_i|[\cO/\cP^{n}: \fp^ibR]^{-\delta_2}\leq |A_i'|[\cO/\cP^{n-i}: bR]^{-\frac{\delta_2}{2}},
\end{align*}
{where in the above, $\mathfrak p^i A_i', \mathfrak p^i a+\mathfrak p^i b R$ are viewed as subsets of $\cO/\cP^n$, and $\mathfrak p^i bR$ as a subgroup of $\cO/\cP^n$.}

Therefore, by our assumption of the validity of Theorem \ref{sumproduct} for $A_i'$, we have  
$$|A_i'+A_i'|+|A_i'\cdot A_i'|>|A_i'|^{1+\delta'},$$
which implies 
\begin{align}\nonumber
|A+A|+|A\cdot A|&>|A_i+A_i|+|A_i\cdot A_i|>|A_i'+A_i'|+ |A_i'\cdot A_i'|/ q^\epsilon \\
&> |A_i'|^{1+\delta'-\frac{\epsilon }{\delta_1}}> |A|^{1+\delta'/2 }.
\end{align}

\subsection{Reduction to a set $A_1$}
Now we assume that $A$ consists of invertible elements, and $A$ satisfies conditions \eqref{1020} and \eqref{2226}. \par

We first recall the following theorem from Bourgain.  

\begin{thmx}[Lemma A1, \cite{Bou08}] \label{1443} Fix $k\in\mathbb Z^+$.  \textit{There exists $C=C(k)>0$ such that the following holds:  Let $S\subset ({\mathcal O/\mathcal P^n})^*$ satisfy 
\begin{align}\label{2321} |S+S|+|S\cdot S|<K|S|
\end{align}
Then there is a subset $S_1\subset S$ such that 
\begin{align}\label{2120}
|S_1|>K^{-2}|S|
\end{align}
\begin{align}\label{2121}
|\sum_kS_1^k-\sum_kS_1^k|<K^C|S_1|.
\end{align}}

\end{thmx}
\begin{remark}
Theorem \ref{1443} was stated in the case $\mathbb Z/q\mathbb Z$.  The proof works for $\mathcal O/\mathcal P^m$ line by line. 
\end{remark}

Now let us assume \eqref{expansion} does not hold.  Apply Theorem \ref{1443} to $A$ with $K=q^{\delta_3}$ and some $k$ to be specified later.  There exists a set $A_1$ satisfying 
\begin{align}\label{2003}
|A_1|>{q^{-2\delta_3}}|A|
\end{align}
\begin{align}\label{2004}
|\sum_kA_1^k-\sum_kA_1^k|<q^{2\delta_3 C(k)}|A_1|.
\end{align}
We will { arrive at} a contradiction if $\delta_3$ is sufficiently small. In fact, we will show by a bounded sum-product of $A_1$ contains a very large subgroup of $\cO/\cP^n$.\par
{We notice that if $\delta_3$ is sufficiently small, namely, if we take \begin{align}\label{2132}
\delta_3<\frac{\epsilon\delta_2}{4},
\end{align}
then for any $a+bR\subset \cO/\cP^n$, $R$ congruential, and $[\cO/\cP^n: bR]>|\cO/\cP^n|^{\epsilon}$, we have 
\begin{align}\label{040}
|A_1\cap (a+bR)|\leq |A\cap (a+bR)|<|A||\cO/\cP^n: bR|^{-\delta_2}<|A_1||\cO/\cP^n: bR|^{-\frac{\delta_2}{2}}
\end{align}
so $A_1$ satisfies the $(\epsilon, \delta_2/2)$ non-concentration condition. }
{ 
\subsection{An overview of Proof of Theorem \ref{sumproduct}}
Let us first sketch some key ideas for the rest of the proof of Theorem \ref{sumproduct}. Let $R_1=\langle A_1\rangle <\cO/\cP^n$ be the ring generated by $A_1$ and $D_0< D_1< D_2\cdots $ be the level groups of $A_1$. Assuming $A_1$ has enough density at each level: $|A_1(\mod p^m)|>p^{d_0m\delta_2}$, we want to show $A_1$ boundedly generate a very large ideal of $\cO/\cP^n$. \par

Case (1): $D_0(\mod (p))=\cO(\mod (p))$. Then $R_1=\cO/\cP^n$, and we can apply Theorem \ref{1053} iteratively to generate a very large subset of $\cO/\cP^n$. Apply Corollary \ref{2346} further to this very large subset then gives a very large ideal of $\cO/\cP^n$. \par

Now we consider Case (2): $D_0(\mod (p))\neq\cO(\mod (p))$. We write $$D_0=D_1=\cdots =D_t\neq D_{t+1}.$$ \par
Case (2a): Suppose $t<M$ with $M$ given in Corollary \ref{324}. Since $M$ is an absolute constant, one can show $A$ boundedly generates $R(\mod p^{t+1})$ by applying Theorem \ref{BKT04} first to $A_1(\mod \cP)$ and pushing level by level. We will show how to do this in our forthcoming rigorous proof.  \par

Case (2b): Suppose $t\geq M$. By Corollary \ref{324}, $R_1(\mod p^t)=\cO_0(\mod p^t)$ for some $\cO_0$. We claim $t$ is bounded by a constant multiple of $n\epsilon$; otherwise, $A_1$ is fully contained in a congruential ring, contradicting the non-concentration assumption of $A_1$. \par
Applying Theorem \ref{1053} and Corollary \ref{324}, we obtain from a bounded sum-product of $A_1$ a set of the form $p^{[\frac{1}{3}t]}A'$, where $A'(\mod p^{[\frac{2}{3}t]})=\cO_0(\mod p^{[\frac{2}{3}t]})$. Now consider the ring $R_2<\cO/\cP^{n- [\frac{1}{3}t] }$ generated by $B= A_1A'$ and the level groups 

$$D_0'  = D_1' = D_2'\cdots =D_{t'}'\neq D_{t'+1}' $$
We must have $D_0' = \cdots =D_{t'}' = \cO_0(\mod p)$ and the first changing time $t'<t$. \par
Since $B(\mod p^{[\frac{2}{3}t]})=  \cO_0(\mod p^{[\frac{2}{3}t]})$ and $t'<t$, we have $$BB+BB (\mod p^{t'})= \cO_0(\mod p^{t'})=R_2(\mod p^{t'}).$$ Take a further sum product then covers the set $R_2(\mod p^{t'+1})$. 

In both Case (2a) and Case (2b), in the end we created elements of the form $p^{n'} A^*$ where $n'$ is compared to $n\epsilon$, and $A^*(\mod p)$ is a strictly larger $p$-group than $D_0$.  Then we run the argument again with $A_1$ replaced by $A_1A^*$ and we keep iterating until reaching a set $p^{n''}A^{**}$ with a tolerable $n''$ (comparable to $n\epsilon$) and $A^{**}(\mod p)=\cO(\mod p)$. The number of iterating times depends only on $[K:\mathbb Q]$. We reduce to Case (1) by working with $A_1A^{**}$. 
 }
\subsection{Proof of Theorem \ref{sumproduct}}
Now we give a formal proof of Theorem \ref{sumproduct}.  {Recall $\cO/\cP=\mathbb F_{p^{d_0}}$.} Let $l_0$ be the largest integer such that $|\pi_{\mathcal P^{l_0}}(A_1)|< p^{\frac{d_0l_0\delta_2}{2}}$.  We must have $l_0\leq n\epsilon$; otherwise, from \eqref{2132}, by taking $R_1=\cO/\cP^n, a=0, b= \fp^{l_0}$ we have $\pi_{\cP^{l_0}}>p^{\frac{d_0l_0\delta_2}{2}}$
 leading to a contradiction.  
Thus $l_0\leq n\epsilon$.  \par 
By the maximality of $l_0$,  for any $l\geq l_0$, there exists $x_l\in A_1$ such that 
\begin{align}
|\pi_{\cP^l}(T_{l_0}(x_l))|>p^{\frac{(l-l_0)d_0\delta_2}{2}},
\end{align}
where  
\begin{align}
T_{l_0}(x):=\{y\in A_1: y\equiv x(\mod \mathcal P^{l_0})\}.
\end{align}
This implies that we have a set $A_2'\subset \mathcal O/\mathcal P^{n-l_0}$ with $\mathfrak p^{l_0}A_2'\subset A_1-A_1$, satisfying for any $1\leq l\leq n-l_0$, 
\begin{align}\label{0604}
|\pi_{\cP^l}(A_2')|>p^{\frac{d_0\delta_2l}{2}}.
\end{align} 
Let $A_2=A_1(\mod \cP^{n-l_0})\cdot A_2'$.  Without loss of generality, we may assume that $1\in A_1$ and $1\in A_2'$, since we use homogeneous polynomial in each set expansion step, {and also multiplying $A_1$ by an invertible does not change the non concentration assumption \eqref{040}}. The property \eqref{0604} is carried over to $A_2$. \par

 Let $R_1$ be the ring generated by $A_2$.  Then, $R_1(\mod \cP)\cong \mathbb{F}_{p^{d_1}}$ is a field for some $d_1| d_0$.  Since $|A_2(\mod \cP)|>p^{d_0\delta_2}$,  Theorem \ref{BKT04} implies that there exists $r_1=r_1({\delta_2})$ such that $$\sum_{r_1} A_2 ^{r_1}-\sum_{r_1} A_2 ^{r_1}(\mod \cP)= R_1(\mod \cP).$$  Indeed, if $|A_2(\mod \cP)|>p^{\frac{9}{10}d_0}$, Corollary \ref{2346} then implies that we can take $r_1=12$.  Otherwise, We iteratively apply Theorem \ref{BKT04} to sets starting from $A_2$, so that at each step the set grows at least by a small power of {$p$}, until {either} we reach a set of size $>p^{\frac{9}{10}d_0}$, for which we can apply Corollary \ref{2346}, or a set $A'$ essentially trapped in a subfield of $K'<\mathbb F_{p^{d_1}}$.  Applying Corollary \ref{2346} again, there exists some $r'=r'(\delta)$, $\sum_{r'}(A')^{r'}\supset K'$. Since as a $K'$-vector space, $\mathbb F_{p^{d_1}}$ has dimension $\leq d_1$, it follows that $A_2^{d_1}$ contains a basis for $\mathbb F_{p^{d_1}}$ as a $K'$-vector space.  Then, $$\sum_{d_1r'}A_2^{d_1}(A')^{r'}(\mod \cP)=\mathbb{F}_{p^{d_1}},$$ which implies the existence of $r_1$ in this case as well. \par 

Let $$A_3=\sum_{r_1}A_2^{r_1}-\sum_{r_1}A_2^{r_1}.$$
Now we consider $R_1(\mod \cP^2)$, we hope to show that there exists $r_2=r_2(\delta_2)$ such that $$\sum_{r_2} A_2 ^{r_2}-\sum_{r_2} A_2^{r_2} = R_1 (\mod \cP^2).$$
Suppose $A_3(\mod \cP^2) = R_1 (\mod \cP^2)$, then we are done by setting $r_2=r_1$. Otherwise, $A_3(\mod \cP^2)$ is not a ring, because $A_2$ generates $R_1$ and $A_2\subset A_3$.  Hence, we can find $x\in A_3\cdot A_3-A_3$, or $A_3+A_3-A_3$, or $A_3-A_3-A_3$ such that $x(\mod \cP^2)\not\in A_3(\mod \cP^2)$.  Let $x'\in A_3$ such that $x'\equiv x(\mod \cP)$, then $\bar{x}=x-x'\in A_2-A_3$ and $\fp || \bar{x}$, recalling it means $\bar{x}=\mathfrak py$ or some invertible $y\in \cO/\cP^n$.   So we have 
 $$ A_3 \bar{x}\subset (A_3^3-A_3^2) \cup (A_3^2+A_2^2-A_3^2) \cup (A_3^2-A_2^2-A_3^2) (\mod \cP^2),$$
and thus if we let $A_3'=A_3+[(A_3^3-A_3^2) \cup (A_3^2+A_2^2-A_3^2) \cup (A_3^2-A_2^2-A_3^2]$, then

 $$|A_3' (\mod \cP^2) |\geq p^{2d_1}.$$

If $A_3'\neq R_1(\mod P^2)$, then by the same argument, we can find $\bar{y}$ from a sum-product set $A_3''$ of $A_3'$ such that $\fp|| \bar{y}$ and $\bar{y} \not \in A_3\bar{x}$, which then implies 
$$|A_3'+A_3''(\mod \cP^2)|\geq p^{3d_1}.$$
We keep iterating.  Since in each step, we increase the lower bound of a sum-product set by a factor of $p^{d_1}$, this process stops at most in $\frac{d_0}{d_1}$ steps, and produces some $r_2>0$ after adding some factor of $A_2$ to homogenise the polynomial (since $1\in A_2$).  \par
By iterating the above argument, we obtain more {generally} 
\begin{lemma}\label{0052}
For any $L$, there exists a constant $r_L=r_L(\delta_2)$ such that 
$$\sum_{r_L}A_2^{r_L}-\sum_{r_L}A_2^{r_L}(\mod \cP^L)=R_1(\mod \cP^L).$$
\end{lemma}
To continue, We divide the analysis into two cases: \par
\medskip

{
\noindent {\bf Case 1: }  $R_1(\mod \cP^e)=\mathcal O/(\cP^e)$, where $e$ is the ramification index of $\cP$.  \par

 In this case, $R_1(\mod \cP^{n-l_0})= \cO(\mod \cP^{n-l_0})$.   We first apply Lemma \ref{0052} with $L=2$ so that a sum-product set of $A_2$ satisfies the assumption of Theorem \ref{1053}.  Then we apply Theorem \ref{1053} iteratively to obtain a number $t_0=t_0(\delta_2)$ such that 
 $$|\sum_{t_0}A_2^{t_0}|>p^{d_0(n-l_0)(1-\epsilon)},$$
 which then gives a number $t_1$, such that 
 \begin{align}\label{223}
 |\sum_{t_1}A_1^{t_1}|> p^{d_0(n-l_0)(1-\epsilon)-l_0t_0}>p^{d_0n(1-(2+t_0)\epsilon)},
 \end{align}
 having in mind of loss of some $p$ factors when creating $A_2$ from $A_1$. So 
 \begin{align}
 \eqref{223}\geq p^{d_0n\frac{\delta_1+1}{2}}
 \end{align}
 if we take
 \begin{align}
 \epsilon=\frac{1-\delta_1}{4(2+t_0)}.
 \end{align}
 We thus arrive at a contradiction to \eqref{2004} if we take $k=t_1$ and 
 \begin{align}
 \delta_3=\min\{\frac{\delta_1}{2C(t_1)}, \frac{d_0\epsilon \delta_2}{4}\}.
 \end{align}
 }
%
%
%
%
%
%

\medskip
\noindent{\bf Case 2:} $R_1(\mod \cP^e)\neq \mathcal O/(\cP^e)$.  

Our strategy in this case is to find a set $\mathfrak p^{l} A'$ from a sum-product set of $A_2$, such that $l$ is small and $A'(\mod \cP^e)=\cO(\mod \cP^e)$ and so we can reduce the analysis to Case 1 situation.    \par
 Let $l_1$ be the largest integer such that 
$$R_1\cap (p^{l_1-1}) (\mod \cP^{el_1})= p^{l_1-1}R_1(\mod \cP^{l_1}).$$ 
We must have $l_1\leq d_0n\epsilon$. Indeed, if $l_1> d_0n\epsilon$, then, Corollary \ref{324} implies $A_1$ sits in a congruential ring $R'$, which is the preimage of $R_1(\mod \cP^{el_1})$ under $\cO/\cP^{n}\mapsto \cO/\cP^{el_1},$ with cardinality $[\cO/\cP^n:R']\geq p^{l_1}>q^{n\epsilon},$
which contradicts \eqref{040}.

If $l_1>M$ for some constant depending only $d$ implied by Theorem \ref{structure}, then $R_1(\mod \cP^{el_1})=\cO'(\mod \cP^{el_1})$, for some $\cO'$ the ring of integers of some field $K'<K$.  In this case, we can apply Theorem \ref{2226} with $n$ replaced by $el_1$, together with Corollary \ref{2346} to obtain $t_2=t_2(\delta_2)$ such that 
\begin{align}\label{0609}
t_2 A_2^{t_2}(\mod \cP^{el_1})\supset p^{[\frac{1}{3}l_1]} \cO' (\mod\cP^{el_1}).
\end{align}
Let $A_4'$ be a set such that $p^{[\frac{1}{3}l_1]}A_4'\subset t_2A_2^{t_2}$ consisting of all representatives of $p^{[\frac{1}{3}l_1]}\cO'(\mod\cP^{el_1})$  on the right hand side of \eqref{0609}. Let $A_4= A_2A_4'$ and $R_2$ be the ring generated by $A_4$.  Then 
\begin{align}\label{0616}
A_4(\mod \cP^{e[\frac{19}{25}l_1]})=\cO'(\mod \cP^{e[\frac{19}{25}l_1]})
\end{align}
If we let $R_2$ be the ring generated by $A_4$ and let $l_2$ be the largest integer such that 
$$R_2\cap (p^{l_2-1}) (\mod \cP^{el_1})= p^{l_2-1}\cO'(\mod \cP^{el_2}).$$ 
We must have $l_2\leq l_1$ since $A_2\subset R_2$ and the level group of $A_2$ already start changing at $l_1$.  From \eqref{0616}, it follows easily that 
\begin{align} \label{0615}
A_4A_4+ A_4 A_4(\mod\cP^{el_1}) =R_2(\mod\cP^{el_1})
\end{align}
Arguing in the same way as Lemma \ref{0052}, we can take a further sum-product to recover one more level for $R_2$: there there exists some $t_3=t_3(d)$, such that 
\begin{align}
\sum_{t_3} A_4^{t_3}(\mod \cP^{e(l_1+1)})=R_2(\mod \cP^{e(l_2+1)})
\end{align}
with \begin{align}\label{0635}R_2\cap (p^{l_2})(\mod \cP^{e(l_2+1)})\supsetneq p^{l_2}R_1(\mod \cP^{e(l_2+1)}).\end{align}
Therefore, if we let $A_5'$ be a set such that $p^{l_2}A_5'\subset \sum_ {t_4}A_4^{t_4}$ consisting of all representatives of $p^{l_2}R_1$  on the right hand side of \ref{0635}, and $A_5=A_5'A_2$, then $A_5$ satisifies

\begin{enumerate}
\item $\langle A_5\rangle (\mod \cP^e)\supsetneq \langle A_2 \rangle (\mod \cP^e)$, 
\item for any $1\leq l\leq  n- 2t_2t_3(l_0+l_1+l_2)$, 
$$\pi_{\mathcal P^l}(A_5)\geq p^{d_0\delta_2 l/2}.$$
\end{enumerate}
{

The point is that, compared to $A_2$, we constructed a set $A_5$ such that the first level group of $\langle A_5\rangle $ is strictly larger $p$ group than that of $\langle A_2\rangle$, while satisfying a same density lower bound as $A_2$ and creating only a manageable power loss of $p$. \par
}

If $l_1\leq M$, one can iteratively apply Lemma \ref{0052} to obtain a set $A_5$ satisfying above properties as well.   \par

 We iterate the above step for $A_5$ as we did for $A_2$. We keep iterating (finished in at most $d_0$ steps) until we obtain a set $A'$ with $$\mathfrak p^{l'}A'\subset \sum_{t_4} A_1^{t_4}-\sum_{t_4} A_1^{t_4}$$ for some $l'=O_{\delta_2, d}(n\epsilon)$, $t_4=O_{\delta_2,d}(1)$, such that 
 $$\pi_{\cP^e}(A')=\cO/(\cP^e),$$
 and for all $1\leq l\leq n-l',$
 $$\pi_{\cP^l}(A')\geq p^{d_0\delta_2l/2}.$$

At this point, we can apply the treatment in Case 1 to $A'$, followed by an application of Corollary \ref{2346}, to obtain a number $t_5$ and a number $l^*$ such that 
\begin{align}\label{2218}
\mathfrak p^{l^*}\cO/\cP^n\subset \sum_{t_5}A_1^{t_5}-\sum_{t_5}A_1^{t_5}
\end{align}
with 
\begin{align}\label{2326}
l^*=O_{\delta_2,d}(n\epsilon)
\end{align} and 
\begin{align}
t_5=O_{\delta_2,d}(1),
\end{align}  

Let $C'$ be the implied constant at \eqref{2326}. Then, 
\begin{align}\label{1309}
|\sum_{t_5}A_1^{t_5}-\sum_{t_5}A_1^{t_5}|>q^{1-C'\epsilon},
\end{align}

\par 
On the other hand,  taking $k=t_5$ in \eqref{2004}, and let $C=C(t_5)$ be the implied constant in \eqref{2004}.  If $K=q^{\delta_3}$ is sufficiently small, from \eqref{2004} we have 
\begin{align}\label{1300}
|\sum_{t_5}A_1^{t_5}-\sum_{t_5}A_1^{t_5}|<q^{1-\delta_1+2C\delta_3} 
\end{align}
If we take 
\begin{align}\label{0620}
\epsilon=\frac{\delta_1}{2C'}
\end{align}
\begin{align}
\delta_3=\min\{\frac{\delta_1}{4C}, \frac{d_0 \delta_2 \epsilon}{4}\},
\end{align}
then \eqref{1300} will contradict \eqref{1309}.  So our assumption that \eqref{expansion} fails is not true, and Theorem \ref{sumproduct} is thus proved.

\newpage

\section{Two Propositions Towards the Proof of Theorem \ref{exp1}}
The goal of this section is to prove Proposition \ref{1441} and Proposition \ref{1536} that will be used to establish Theorem \ref{exp1}.  \par
We first recall the following two known results from combinatorics.  
\begin{thmx}[Balog-Szemer\'edi-Gowers theorem \cite{Schoen2015}]\label{com1}

\textit{Let $\mathcal{G}$ be an abelian group. Define the additive energy of a set $A \subset \mathcal{G}$ :
$$
E^{\boxplus}(A)=\left|\left\{\left(a_{1}, a_{2}, a_{3}, a_{4}\right) \in A^{4}: a_{1}+a_{2}=a_{3}+a_{4}\right\}\right| .
$$
Let $A \subset \mathcal{G}$, $\delta>0$ a constant such that \begin{equation}
\label{B7.26.3}
E^{\boxplus}(A)= \delta|A|^{3}.  \end{equation} Then there exists $A^{\prime} \subset A$ such that 
\begin{equation}
\label{B7.27.1}
\left|A^{\prime}\right|=\Omega(\delta|A|)
\end{equation}
and
\begin{equation}
\label{B7.28.1}
\left|A^{\prime}-A^{\prime}\right|=O\left(\delta^{-4}\left|A^{\prime}\right|\right).
\end{equation} }
\end{thmx}

\begin{lemma}[Sum-set estimates \cite{Nathanson1996}]\label{com2} 
Let $\mathcal{G}$ be an abelian group. Let $A, B\subset \mathcal{G}$ such that $|A+B|\leq K|A|$, then for any $k, l\in \mathbb{Z}_{+}$,
\begin{equation}
|\sum_kB-\sum_lB|\leq K^{k+l}|B|.
\end{equation}
\end{lemma}

Now prove the following proposition, which crucially relate certain distribution properties of $\mu$ and sum-product expansion. 

\begin{proposition}\label{1441}
 Let $\mu$ be a probability measure on $\cO/\cP^n$. Write $\cO/\cP=\mathbb F_{p^{d_0}}$ for some $d_0\leq d$, $q=|\cO/\cP^n|$, and
\begin{equation}
\label{B7.10.0}
\phi=q\left(\mu * \mu_{-}\right).
\end{equation}
Let $1\gg \epsilon>\tau>0$.
Then one of the following alternatives holds.
\begin{itemize}
    \item[(i)]  \begin{equation}\label{B7.6.1}
    \sum_{y \in \cO/\cP^n}\sum_{\xi\in \hat{\cO/\cP^n}}|\hat{\mu}(\xi)|^{2}|\hat{\mu}(y \xi)|^{2} \mu(y)<q^{-\tau} \sum_{\xi \in \hat{\cO/\cP^n}}|\hat{\mu}(\xi)|^{2}  \end{equation}
    where $y\xi(x)=\xi(xy)$.
    \item[(ii)] The measure $\mu$ is concentrated on some small congruence class, i.e.
    \begin{equation}\label{B7.6.2}
    \max\limits_{y \in \cO/\cP^n} \mu\left(\pi_{m}^{-1}(y)\right)>c p^{-2d_0\smallerpower \tau / \epsilon}
    \end{equation}
    for some $\smallerpower>\epsilon \startingpower$ and constant $c>0$ independent of $q$.
    \item[(iii)] There is a subset $\bar{S} \subset (\cO/\cP^n)^{*}$ with small the sum-product expansion, namely, for some $C>0$ independent of $q$, $\bar{S}$ satisfies
\begin{equation}
\label{B7.6}
|\bar{S}|\cdot\left(\sum_{\xi \in \hat{\cO/\cP^n}}|\hat{\mu}(\xi)|^{2}\right) <q^{1+2\epsilon},
\end{equation}
\begin{equation}
\label{B7.7}
|\bar{S}+\bar{S}|+|\bar{S} \cdot \bar{S}| <q^{C \epsilon}|\bar{S}|,
\end{equation}
and there exists some ${m'}<2\epsilon \startingpower$ such that
\begin{equation}
\label{B7.38.1}
|\mathfrak{p}^{m'} \bar{S}|=|\bar{S}|,
\end{equation}
\begin{equation}
\label{B7.17.1}
\phi(x)>q^{-2\epsilon}\phi(0) \text{ for all } x\in \mathfrak{p}^{\smallerpower^\prime} \bar{S}
\end{equation}
and
\begin{equation}
\label{B7.8.2}  \left(\mu * \mu_{-}\right)\left(\mathfrak{p}^{\smallerpower^{\prime}} \bar{S}\right)>q^{-C\epsilon}.
\end{equation}
\end{itemize}
\end{proposition}
\begin{proof}[Proof of Proposition \ref{1441}]
Assume $(\ref{B7.6.1})$ fails, so that
\begin{equation}
\label{B7.9} \sum_{y \in \cO/\cP^n}  \sum_{\xi \in \hat{\cO/\cP^n}} |\hat{\mu}(\xi)|^{2}|\hat{\mu}(y \xi)|^{2} \mu(y) \geq q^{-\tau} \sum_{\xi \in \hat{\cO/\cP^n}}|\hat{\mu}(\xi)|^{2}
\end{equation}
Fix $x\in \cO/\cP^n$. We compute
\begin{align}
\label{B7.10.1}
\sum_{\xi\in \hat{\cO/\cP^n}}|\hat{\mu}(\xi)|^{2} \xi(x)&=\sum_{\xi\in \hat{\cO/\cP^n}}\left(\sum\limits_{t_1\in \cO/\cP^n} \xi(t_1)\mu(t_1)\right)\left(\overline{\sum\limits_{t_2\in \cO/\cP^n} \xi(t_2)\mu(t_2)}\right) \xi(x) \nonumber \\
&=\sum_{\chi\in \hat{\cO/\cP^n}}\sum_{ t_1, t_2\in \cO/\cP^n}\xi(t_1-t_2+x)\mu(t_1)\mu(t_2).
\end{align}

We have 
\begin{equation}
\label{B7.10.2}
\sum_{t\in \cO/\cP^n}\xi(t)=\begin{cases}0, &\text{ if }\xi\neq 0\\
q, &\text{ if }\xi= 0.\end{cases}\end{equation}
Therefore, by $(\ref{B7.10.1})$ and $(\ref{B7.10.2})$ we have
\begin{align}
\label{B7.10}
\sum_{\xi\in \hat{\cO/\cP^n}}|\hat{\mu}(\xi)|^{2} \xi(x)&=\sum_{t_1\in \cO/\cP^n}q\mu(t_1)\mu(t_1+x)\nonumber\\
&=q\sum_{\substack{y, z\in \cO/\cP^n\\z-y=x}}\mu(y)\mu(z)=q\sum_{\substack{y, z\in \cO/\cP^n\\z+y=x}}\mu(-y)\mu(z)\nonumber\\
&=q\left(\mu * \mu_{-}\right)(x)=\phi(x).
\end{align}
It follows that
\begin{equation}
\label{B7.11}
\phi(0)=\sum_{\xi\in \widehat{\cO/\cP^n}}|\hat{\mu}(\xi)|^{2}=\max _{x\in \cO/\cP^n} \phi(x),
\end{equation}
 and
\begin{align*}
\sum_{x, y \in \cO/\cP^n} \phi(x) \phi(x y) \mu(y)&=\sum_{x, y \in \cO/\cP^n}\left(\sum_{\xi_1\in \hat{\cO/\cP^n}}|\hat{\mu}(\xi_1)|^{2} \xi_1(x)\right)\left(\sum_{\xi_2\in \hat{\cO/\cP^n}}|\hat{\mu}(\xi_2)|^{2} \xi_2(xy)\right)\mu(y)\\
&=\sum_{x, y\in \cO/\cP^n}\sum_{\xi_1, \xi_2 \in \hat{\cO/\cP^n}}|\hat{\mu}(\xi_1)|^{2}|\hat{\mu}(\xi_2)|^{2}(\xi_1+y\xi_2)(x)\mu(y)\\
&=q\sum_{y\in \cO/\cP^n}\sum_{\xi_2 \in \hat{\cO/\cP^n}}|\hat{\mu}(y\xi_2)|^{2}|\hat{\mu}(\xi_2)|^{2}\mu(y)
\end{align*}

Hence, $(\ref{B7.9})$ may be reformulated as
\begin{equation}
\label{B7.12}
\sum_{x, y \in \cO/\cP^n} \phi(x) \phi(x y) \mu(y) \geq q^{1-\tau} \phi(0).
\end{equation}

Write $\cO/\cP^n= \sqcup_{0\leq v\leq n}R_{v}$, where 
\begin{equation}\label{1129}
R_v=\begin{cases}\{x\in \cO/\cP^n: \fp^v|| x\}&\text{ if }v\leq n\\ \{0\}&\text{ if }v=n\end{cases}.
\end{equation}
  There are $0 \leq v_{1}\leq \startingpower, 0 \leq v_{2}\leq \startingpower$ such that
\begin{equation}
\label{B7.14}
\sum_{x \in R_{v_{1}}, y \in R_{v_{2}}} \phi(x) \phi(x y) \mu(y)>\frac{1}{\startingpower^{2}} q^{1-\tau} \phi(0)>q^{1-\tau-} \phi(0) .
\end{equation}
Here for notational simplicity, we use Bourgain's notation.  The second inequality of \eqref{B7.14} means $\frac{1}{n^2}q^{1-\tau}\phi(0)>q^{1-\tau-\varepsilon}\phi(0)$ for arbitrarily small $\varepsilon$ when $q$ sufficiently large. 

Assume that $(\ref{B7.6.2})$ also fails. Thus
\begin{equation}
 \label{B7.15}  
\max\limits_{x \in \cO/\cP^n} \mu\left(\pi_{\smallerpower}^{-1}(x)\right)\leq c p^{-2d_0\smallerpower \tau / \epsilon} \text{ for all }  \smallerpower>\epsilon \startingpower
\end{equation}
Since $(\ref{B7.14})$ implies that
\begin{align*}
\phi(0)\mu(R_{v_{2}})q &{\geq} \sum_{x \in R_{v_{1}}} \phi(x) \phi(0)\mu(R_{v_{2}})\geq\sum_{x \in R_{v_{1}}, y \in R_{v_{2}}} \phi(x) \phi(0) \mu(y)\\
&\geq \sum_{x \in R_{v_{1}}, y \in R_{v_{2}}} \phi(x) \phi(x y) \mu(y)>q^{1-\tau-} \phi(0),
\end{align*}
we have
\begin{equation}
\label{B7.15.1}
\mu\left(\pi_{v_2}^{-1}(0)\right)\geq \mu\left(R_{v_{2}}\right) > q^{-\tau-}>cq^{-2\tau}=c p^{-2d_0\startingpower\tau}
\end{equation}
for sufficiently large $q$. If $v_2>\epsilon \startingpower$, then $c p^{-2d_0\startingpower\tau}>c p^{-2d_0\tau v_2/\epsilon}$ and 
$(\ref{B7.15.1})$ will contradict $(\ref{B7.15})$. Thus we know $v_2\leq \epsilon \startingpower.$
Also,
\begin{align}
\label{B7.15.2}
q^{1-\tau-}&<\sum_{x \in R_{v_{1}}} \phi(x)=q\left(\mu * \mu_{-}\right)\left(R_{v_{1}}\right)=q\sum\limits_{\substack{y,z\\y-z\in R_{v_1}}}\mu(y)\mu(z) \\&=q\sum\limits_{x \in \cO/\cP^{v_1}}\mu\left(\pi_{v_1}^{-1}(x)\right)^2\nonumber \leq q \cdot \max_{x \in \cO/\cP^{v_1}} \mu\left(\pi_{v_1}^{-1}(\xi)\right).\end{align}
Assume $v_{1}>\epsilon \startingpower$. It follows from $(\ref{B7.15})$ and $(\ref{B7.15.2})$ that
$$
q^{-\tau-}<\max_{\xi \in \cO/\cP^{v_1}} \mu\left(\pi_{v_1}^{-1}(\xi)\right)<c p^{-2d_0v_1 \tau / \epsilon}<cp^{-2d_0\tau \startingpower}=cq^{-2\tau},
$$
which leads to a contradiction for large $q$. Therefore,
\begin{equation}
\label{B7.16} 
v_{1}, v_{2} \leq \epsilon \startingpower.
\end{equation}
Notice that for $y \in R_{v_2}, x_1, x_2\in \cO/\cP^n$, we have
\begin{equation}
\label{B7.16.1}
{ x_1y=x_2y \Rightarrow \mathfrak p^{v_2} || x_1-x_2.}
\end{equation}
Therefore, at most $p^{d_0v_2}$ different $x_i$ give the same $x_iy$. By $(\ref{B7.11.1}), (\ref{B7.16})$, for $y \in R_{v_2}$ we have
\begin{equation}
\label{B7.16.2}\sum_{x \in \cO/\cP^n} \phi(x y) \leq p^{d_0v_{2}} \sum_{x \in \cO/\cP^n} \phi(x) \leq q^{1+\epsilon}.
\end{equation}
\\
We proceed with the proof of Proposition \ref{1441} by defining 
\begin{equation}
\label{B7.17}
S=\{x \in \cO/\cP^n|\phi(x)>q^{-2\epsilon}\phi(0)\}.
\end{equation}
Recalling $(\ref{B7.10.0})$, we know
\begin{equation}
\label{B7.11.1}
\sum_{x \in \cO/\cP^n} \phi(x)=q.
\end{equation}
It follows from $(\ref{B7.11.1})$ that
\begin{equation}
\label{B7.18}
|S|<q^{1+2\epsilon}\phi(0)^{-1}.
\end{equation}
By (\ref{B7.14}), 
\begin{align*}
q^{1-\tau-} \phi(0)&< \sum_{x \in R_{v_{1}}, y \in R_{v_{2}}} \phi(x) \phi(x y) \mu(y)\\
&=(\sum_{x \in R_{v_{1}}\cap S, y \in R_{v_{2}}, xy\in S}+\sum_{x \in R_{v_{1}}\cap S, y \in R_{v_{2}}, xy\notin S}+\sum_{x \in R_{v_{1}}, x \notin S, y \in R_{v_{2}}} )
\left(\phi(x) \phi(x y) \mu(y)\right)\\
&\leq\sum_{x \in R_{v_{1}}\cap S, y \in R_{v_{2}}, xy\in S}\phi(x) \phi(x y) \mu(y)\\
&+\sum_{x \in R_{v_{1}}\cap S, y \in R_{v_{2}}, xy\notin S}q^{-2\epsilon}\phi(0)\phi(x)\mu(y)+\sum_{x \in R_{v_{1}}, x \notin S, y \in R_{v_{2}}}q^{-2\epsilon}\phi(0)\phi(xy)\mu(y)\\
&\leq \sum_{x \in R_{v_{1}}\cap S, y \in R_{v_{2}}, xy\in S}\phi(x) \phi(xy) \mu(y)\\
&+q^{-2\epsilon}\phi(0)\sum_{x \in \cO/\cP^n, y \in R_{v_{2}}}\phi(x) \mu(y)+q^{-2\epsilon}\phi(0)\sum_{x \in \cO/\cP^n, y \in R_{v_{2}}}\phi(xy) \mu(y)\\
&\leq \sum_{x \in R_{v_{1}}\cap S, y \in R_{v_{2}}, xy\in S}\phi(x) \phi(x y) \mu(y)+q^{-2\epsilon}\phi(0)(q^{1+\epsilon}+q),
\end{align*}
Hence, by $\epsilon>\tau$, 
\begin{align}
\quad|S| \geq\left|S \cap R_{v_{1}}\right| &\geq \sum\limits_{y \in R_{v_{2}}}\bigg|\left\{x: x \in S \cap R_{v_{1}}, x y \in S\right\}\bigg| \mu(y)\nonumber\\
&=\phi(0)^{-2}\sum\limits_{x \in R_{v_{1}}\cap S, y \in R_{v_{2}}, xy\in S}\phi(0)^{2}\mu(y)\nonumber\\
&\geq\phi(0)^{-2} \sum_{x \in R_{v_{1}}\cap S, y \in R_{v_{2}}, xy\in S}\phi(x) \phi(x y) \mu(y)\nonumber\\
&>\phi(0)^{-2}(q^{1-\tau-} \phi(0)-q^{-2\epsilon}\phi(0)(q^{1+\epsilon}+q))\nonumber\\
&>q^{1-2 \tau} \phi(0)^{-1}.\label{B7.19}
\end{align}
Next, defining
\begin{equation}
\label{B7.20}
\Lambda=\left\{y \in R_{v_{2}}\bigg|\left|\left\{x \in S \cap R_{v_{1}}: x y \in S\right\} \right|\ > q^{1-3\tau} \phi(0)^{-1}\right\},
\end{equation}
we have by $(\ref{B7.19})$ that
\begin{align*}
q^{1-2 \tau} \phi(0)^{-1}&<\sum\limits_{y \in R_{v_{2}}}\bigg|\left\{x \in S \cap R_{v_{1}}, x y \in S\right\}\bigg| \mu(y)\\
&=\sum\limits_{y \in \Lambda}\bigg|\left\{x \in S \cap R_{v_{1}}, x y \in S\right\}\bigg| \mu(y)+\sum\limits_{y \in R_{v_{2}}, y\notin \Lambda}\bigg|\left\{x \in S \cap R_{v_{1}}, x y \in S\right\}\bigg| \mu(y)\\
&\leq |S|\mu(\Lambda)+q^{1-3\tau} \phi(0)^{-1},
\end{align*}
and it follows that 
\begin{equation}
\label{B7.21.1}
\mu(\Lambda)|S|>q^{1-2 \tau} \phi(0)^{-1}-q^{1-3\tau} \phi(0)^{-1}>q^{1-2 \tau-} \phi(0)^{-1}>q^{1-2 \epsilon} \phi(0)^{-1}>q^{-4 \epsilon}|S|
\end{equation}
by $(\ref{B7.18})$. Hence
\begin{equation}
\label{B7.21}
\mu(\Lambda)>q^{-4 \epsilon}.
\end{equation}
On the other hand, by Cauchy-Schwarz inequality and Parseval identity
$$
\begin{aligned}
q^{-4 \epsilon} &<\mu(\Lambda)=\sum_{x \in \Lambda} \mu(x) \leq|\Lambda|^{1 / 2}\left(\sum_{x \in \Lambda} \mu(x)^{2}\right)^{1 / 2}\leq|\Lambda|^{1 / 2}\left(\sum_{x \in \cO/\cP^n} \mu(x)^{2}\right)^{1 / 2}\\
&=\left(\frac{|\Lambda| \phi(0)}{q}\right)^{1 / 2}\stackrel{(\ref{B7.18})}{<}\left(q^{2 \epsilon} \frac{|\Lambda|}{|S|}\right)^{1 / 2}
\end{aligned}
$$
implying
\begin{equation}
\label{B7.22}
|\Lambda|>q^{-10 \epsilon}|S| .
\end{equation}

Recalling $(\ref{B7.20})$, take $\Lambda_{*} \subset (\cO/\cP^n)^{*}$ such that
\begin{equation}
\label{B7.23}
\Lambda=\mathfrak{p}^{v_{2}} \Lambda_{*} \text { and }|\Lambda|=\left|\Lambda_{*}\right| .
\end{equation}
We also let
\begin{equation}
\label{B7.24}
S_{*}=\{x \in (\cO/\cP^n)^{*} \mid \mathfrak{p}^{n^\prime} x \in S \text{ for some } n^{\prime}\leq 2\epsilon \startingpower \}.
\end{equation}

On one hand, by the same argument as $(\ref{B7.16.1})$, we know for $x\in R_v$, then $x=\mathfrak{p}^v x_0$ for at most $p^{d_0v}$ different choices of $x_0\in (\cO/\cP^n)^*$. It follows that
\begin{equation}
\label{B7.25.1}
\left|S_{*}\right| \leq \sum_{v=0}^{[ 2\epsilon \startingpower]}p^{d_0v} |S\cap R_v|\leq p^{2\epsilon \startingpower} \sum_{v=0}^{[ 2\epsilon \startingpower]} |S\cap R_v|\leq q^{2\epsilon}|S|.
\end{equation}

On the other hand, from $(\ref{B7.18})$ and $(\ref{B7.19})$ we obtain
\begin{equation}
\label{B7.25.2}
\left|S_{*}\right|{>}q^{-4 \epsilon}|S|.
\end{equation}
Hence the size of $S$ and $S_*$ differ up to a small power of $q$:
\begin{equation}
\label{B7.25}
q^{-4 \epsilon}|S|<\left|S_{*}\right|\leq q^{2\epsilon}|S|.
\end{equation}

We apply Theorem \ref{com1} to estimate $E^\boxtimes(S_*)$ the multiplicative energy of $S_*$.  We first estimate $\left|S_{*} \cap y^{-1} S_{*}\right|$ for a fixed $y\in \Lambda_*$.  By $(\ref{B7.20})$ and  $(\ref{B7.23})$, 
\begin{equation}
\label{B7.26.1}
\bigg|\left\{x \in S \cap R_{v_{1}}: x \mathfrak{p}^{v_2}y \in S\right\} \bigg|\ > q^{1-3\tau} \phi(0)^{-1}.
\end{equation}
For $x\in S \cap R_{v_{1}}, x \mathfrak{p}^{v_2}y \in S$, write $x=\mathfrak{p}^{v_1}x_0$ for $x_0\in (\cO/\cP^n)^*$, then $\mathfrak{p}^{v_1+v_2}x_0y \in S$ with $x_0y\in (\cO/\cP^n)^*$ and $v_1+v_2\leq 2\epsilon \startingpower$ by $(\ref{B7.16})$. Therefore, $x_0\in S_*$, $x_0y\in S_*$ and thus $x_0\in S_{*} \cap y^{-1} S_{*}$. It is clear that different $x\in S \cap R_{v_{1}}$ with $x \mathfrak{p}^{v_2}y \in S$ correspond to different $x_0\in S_{*} \cap y^{-1} S_{*}$.  \par 
By $(\ref{B7.26.1})$, $(\ref{B7.18})$ and $(\ref{B7.25.1})$, we know
\begin{equation}
\label{B7.26}
\left|S_{*} \cap y^{-1} S_{*}\right|>q^{1-3 \tau} \phi(0)^{-1}>q^{-3 \tau-2\epsilon}|S|>q^{-5 \epsilon}\left|S\right| \geq q^{-7 \epsilon}\left|S_{*}\right| 
\end{equation}
for $y\in \Lambda_*$. \par
Define for $z\in (\cO/\cP^n)^*$,
\begin{equation}
e^\boxtimes(z)=e_{S_*}^\boxtimes(z)=\left|\{(a,b)\in S_*\times S_*\big |ab^{-1}=z\}\right|.
\end{equation}
Then
\begin{align}
\label{B7.26.2}
E^{\boxtimes}(S_*)&=\left|\left\{\left(a_{1}, a_{2}, a_{3}, a_{4}\right) \in S_*^{4}: a_{1}a_{2}=a_{3}a_{4}\right\}\right|\nonumber\\
&=\left|\left\{\left(a_{1}, a_{2}, a_{3}, a_{4}\right) \in S_*^{4}: a_{1}a_{3}^{-1}=a_{4}a_{2}^{-1}\right\}\right|\nonumber\\
&=\sum_{z\in (\cO/\cP^n)^*}e^\boxtimes(z)^2.
\end{align}
Observe that for $x\in S_{*} \cap y^{-1} S_{*}$, we have $x\in S_{*}$, $xy\in S_{*}$, and thus $(xy, x)$ gives a solution to $ab^{-1}=y$. It's clear that different $x$ correspond to different $(xy, x)$, so 
\begin{equation}
e^\boxtimes(y)\geq \left|S_{*} \cap y^{-1} S_{*}\right|> q^{-7 \epsilon}\left|S_{*}\right|.
\end{equation}
Therefore, by $(\ref{B7.26.2})$, $(\ref{B7.23})$, $(\ref{B7.22})$ and $(\ref{B7.25.1})$
\begin{align} \nonumber
E^\boxtimes(S_{*})&\geq \sum_{z\in (\cO/\cP^n)^*}e^\boxtimes(z)^2\geq \sum_{y\in \Lambda^*}e^\boxtimes(y)^2> |\Lambda^*|(q^{-7 \epsilon}\left|S_{*}\right|)^2\\ &\geq q^{-24\epsilon}\left|S_{*}\right|^2\left|S\right|\geq q^{-26\epsilon}\left|S_{*}\right|^3.
\end{align}
Then we can apply Theorem \ref{com1} to the set $S_{*} \subset (\cO/\cP^n)^{*}$ by taking some $\delta>q^{-26\epsilon}$ in $(\ref{B7.26.3})$ and obtain $S_{*}^{\prime} \subset S_{*}$ such that
\begin{align}
\label{B7.27.2}
\left|S_{*}^{\prime}\right| &=\Omega(\delta\left|S_{*}\right|), \\
\label{B7.28.2}
\left|S_{*}^{\prime} \cdot (S_{*}^{\prime})^{-1}\right| &=O(\delta^{-4}\left|S_{*}^{\prime}\right|).
\end{align}
Then from $(\ref{B7.28.2})$, together with an application of Lemma \ref{com2} to $A=S_*'$, $B=(S_*')^{-1}$, $k=l=1$, we get 

\begin{equation}
\label{B7.28.3}
\left|S_{*}^{\prime} \cdot S_{*}^{\prime}\right| =O(\delta^{-8}\left|S_{*}^{\prime}\right|).
\end{equation}
In other words, 
\begin{align}
\left|S_{*}^{\prime}\right| &>q^{-27 \epsilon}\left|S_{*}\right|, \label{B7.27}\\
\left|S_{*}^{\prime} \cdot S_{*}^{\prime}\right| &<q^{209 \epsilon}\left|S_{*}^{\prime}\right|\label{B7.28}
\end{align}
for sufficiently large $q$. \par
Take a large subset $S_*''$ of $S_*'$ such that\begin{equation}
\label{B7.29}
S^{\prime}=\mathfrak{p}^{m^{\prime}} S_{*}'' \subset S, \quad \text { for some } m'\leq 2 \epsilon \startingpower,
\end{equation}
and
\begin{equation}
\label{B7.30}
\left|S^{\prime}\right|=|S_*''|\geq \frac{1}{2\epsilon \startingpower}\left|S_{*}^{\prime}\right|>q^{-28 \epsilon}|S| .
\end{equation}

Next, we pass to the additive property.
From $(\ref{B7.10.0})$ and $(\ref{B7.17})$,
$$
q \cdot \sum_{x \in S^{\prime}} \sum_{y\in \cO/\cP^n} \mu(x+y) \mu(y)=\sum_{x \in S^{\prime}}\phi(x)>q^{-2 \epsilon} \phi(0)\left|S^{\prime}\right|.
$$
Hence, by Cauchy-Schwarz inequality and Parseval's identity,
\begin{equation}
\label{B7.31}
\sum_{y\in \cO/\cP^n}\left(\sum_{x \in S^{\prime}} \mu(x+y)\right)^{2}>\frac{\left(q^{-1-2 \epsilon} \phi(0)\left|S^{\prime}\right|\right)^{2}}{\sum\limits_{y\in \cO/\cP^n} \mu(y)^{2}}=\frac{q^{-1-4 \epsilon} \phi(0)^2\left|S^{\prime}\right|^{2}}{\sum\limits_{\xi\in \hat{\cO/\cP^n}} |\hat{\mu}(\xi)|^{2}}=q^{-1-4 \epsilon} \phi(0)\left|S^{\prime}\right|^{2} .
\end{equation}
The left-hand-side of $(\ref{B7.31})$ equals
$$
\sum_{x_{1}, x_{2} \in S^{\prime}} \sum_{y\in \cO/\cP^n} \mu\left(x_{1}+y\right) \mu\left(x_{2}+y\right)=q^{-1} \sum_{x_{1}, x_{2} \in S^{\prime}} \phi\left(x_{1}-x_{2}\right) .
$$
Hence
\begin{equation}
\label{B7.32}
\sum_{x_{1}, x_{2} \in S^{\prime}} \phi\left(x_{1}-x_{2}\right)>q^{-4 \epsilon} \phi(0)\left|S^{\prime}\right|^{2}.
\end{equation}
Define
\begin{equation}
\label{B7.33}
\tilde{S}=\left\{x \in \cO/\cP^n \mid \phi(x)>\frac{1}{2}q^{-4 \epsilon} \phi(0)\right\}\supset S.
\end{equation}
We again see from $(\ref{B7.11.1})$ that
\begin{equation}
\label{B7.34}
|\tilde{S}|<q^{1+4 \epsilon} \phi(0)^{-1} \stackrel{(\ref{B7.19})}{<}q^{6\epsilon}|S| \stackrel{(\ref{B7.30})}{<}q^{34 \epsilon}\left|S^{\prime}\right|.
\end{equation}
Assume that
\begin{equation}
\label{B7.35.1}
\left|\left\{\left(x_{1}, x_{2}\right) \in S^{\prime} \mid x_{1}-x_{2} \in \tilde{S}\right\}\right|\leq\frac{1}{2}q^{-4 \epsilon}\left|S^{\prime}\right|^{2}.
\end{equation}
We compute
\begin{align}
&\sum_{x_{1}, x_{2} \in S^{\prime}} \phi\left(x_{1}-x_{2}\right)=\sum_{\substack{x_{1}, x_{2} \in S^{\prime},\\x_{1}-x_{2} \in \tilde{S}}}\phi\left(x_{1}-x_{2}\right)+\sum_{\substack{x_{1}, x_{2} \in S^{\prime},\\x_{1}-x_{2} \notin \tilde{S}}}\phi\left(x_{1}-x_{2}\right)\nonumber\\
\leq & \left|\left\{\left(x_{1}, x_{2}\right) \in S^{\prime} \mid x_{1}-x_{2} \in \tilde{S}\right\}\right|\cdot \phi(0)+\left|\left\{\left(x_{1}, x_{2}\right) \in S^{\prime} \mid x_{1}-x_{2} \notin \tilde{S}\right\}\right|\cdot \frac{1}{2}q^{-4 \epsilon}\phi(0)\nonumber\\
\leq & \frac{1}{2}q^{-4 \epsilon}\left|S^{\prime}\right|^{2}\cdot\phi(0)+\left|S^{\prime}\right|^{2}\cdot\frac{1}{2}q^{-4 \epsilon}\phi(0)\nonumber\nonumber\\
= & q^{-4 \epsilon} \phi(0)\left|S^{\prime}\right|^{2},
\end{align}
which contradicts $(\ref{B7.32})$. Therefore,
\begin{equation}
\label{B7.35}
\left|\left\{\left(x_{1}, x_{2}\right) \in S^{\prime} \mid x_{1}-x_{2} \in \tilde{S}\right\}\right|>\frac{1}{2}q^{-4 \epsilon}\left|S^{\prime}\right|^{2}.
\end{equation}
Note that $(\ref{B7.34})$ and $(\ref{B7.35})$ suggest that we can apply the Balog-Szemer\'edi-Gowers theorem to $S^\prime$ as a subset of $\cO/\cP^n$ as an additive group.

Let $e^\boxplus(z)=e_{S'}^\boxplus(z)=\{(x_1,x_2)\in S': x_1-x_2=z\}$.  By $(\ref{B7.26.2})$ and Cauchy-Schwarz inequality, the additive energy
\begin{align}
E^\boxplus(S^\prime)&= \sum_{z\in \cO/\cP^n}e^\boxplus(z)^2\geq \sum_{z\in \tilde{S}}e^\boxplus(z)^2\geq \frac{1}{|\tilde{S}|}\left(\sum\limits_{z\in \tilde{S}}e^\boxplus(z)\right)^2\\&=\frac{1}{|\tilde{S}|}\left|\left\{\left(x_{1}, x_{2}\right) \in S^{\prime} \mid x_{1}-x_{2} \in \tilde{S}\right\}\right|^2\nonumber\\
&\stackrel{(\ref{B7.35})}{>}\frac{1}{|\tilde{S}|}\left(\frac{1}{2}q^{-4 \epsilon}\left|S^{\prime}\right|^{2}\right)^2=\frac{1}{4|\tilde{S}|}q^{-8 \epsilon}\left|S^{\prime}\right|^{4}\nonumber\\
&\stackrel{(\ref{B7.34})}{>}\frac{1}{4q^{34\epsilon}|S^{\prime}|}q^{-8 \epsilon}\left|S^{\prime}\right|^{4}>q^{-43 \epsilon}\left|S^{\prime}\right|^{3}
\end{align}
for sufficiently large $q$. Hence $(\ref{B7.26.3})$ is satisfied with $A=S'$ and some $\delta=q^{-43\epsilon}$. Theorem \ref{com1} gives $S^{\prime \prime} \subset S^{\prime}$ satisfying
\begin{equation}
\label{B7.36}
\left|S^{\prime \prime}\right|>q^{-44 \epsilon}\left|S^{\prime}\right|\stackrel{(\ref{B7.30})}{>}q^{-76 \epsilon}|S|
\end{equation}
and
\begin{equation}
\label{B7.37}
\left|S^{\prime \prime}+S^{\prime \prime}\right|<q^{345 \epsilon}\left|S^{\prime \prime}\right|.
\end{equation}
Recalling $(\ref{B7.29})$, take a subset $\bar{S} \subset S_{*}^{\prime}$ such that
\begin{equation}
\label{B7.38}
|\bar{S}|=\left|S^{\prime \prime}\right| \quad \text { and } \quad S^{\prime \prime}=\mathfrak{p}^{\smallerpower^{\prime}} \bar{S},
\end{equation}
which is $(\ref{B7.38.1})$.
By the same argument as $(\ref{B7.16.1})$, we know
\begin{equation}
|\bar{S}+\bar{S}| \leq p^{d_0\smallerpower^{\prime}}\left|S^{\prime \prime}+S^{\prime \prime}\right|.
\end{equation}
Since $\smallerpower^{\prime}\leq 2 \epsilon \startingpower$,
\begin{equation}
\label{B7.39}
|\bar{S}+\bar{S}| \leq q^{2 \epsilon}\left|S^{\prime \prime}+S^{\prime \prime}\right| \stackrel{(\ref{B7.37})}{<} q^{347 \epsilon}|\bar{S}|.
\end{equation}
From $(\ref{B7.28})$,
\begin{equation}
\label{B7.40}
|\bar{S} . \bar{S}| \leq\left|S_{*}^{\prime} . S_{*}^{\prime}\right|<q^{209 \epsilon}|S_*^{\prime}|\leq q^{209 \epsilon}|S_*|\stackrel{(\ref{B7.25.1})}{<}q^{211 \epsilon}|S| \stackrel{(\ref{B7.36})}{<}q^{287 \epsilon}|\bar{S}| .
\end{equation}
Thus $\bar{S} \subset (\cO/\cP^n)^{*}$ satisfies $(\ref{B7.7})$ with $C>347$.
Since $|\bar{S}| = \left|S^{\prime \prime}\right| \leq \left|S^{\prime }\right| \leq |S|,(\ref{B7.11}),(\ref{B7.18})$ imply
\begin{equation}
\label{B7.41}
|\bar{S}|\left(\sum_{\xi \in \hat {\cO/\cP^n}}|\hat{\mu}(\xi)|^{2}\right) \leq|S| \phi(0)<q^{1+2 \epsilon},
\end{equation}
which is $(\ref{B7.6})$.
By $(\ref{B7.29})$,
$$
\mathfrak{p}^{\smallerpower^{\prime}} \bar{S}=S^{\prime \prime}\subset S^{\prime}= \mathfrak{p}^{\smallerpower^{\prime}} S_{*}^{\prime} \subset S,
$$
which gives $(\ref{B7.17.1})$ by $(\ref{B7.17})$. \par
Finally,
\begin{align}
\label{B7.8.1}
\left(\mu * \mu_{-}\right)\left(\mathfrak{p}^{\smallerpower^{\prime}} \bar{S}\right) \geq\left|\mathfrak{p}^{\smallerpower^{\prime}} \bar{S}\right| \cdot q^{-1-2 \epsilon} \phi(0)=|S^{\prime \prime}|q^{-1-2 \epsilon} \phi(0)\stackrel{(\ref{B7.36})}{>}q^{-1-78 \epsilon} \phi(0)|S|\stackrel{(\ref{B7.19})}{>} q^{-80\epsilon},
\end{align}
which gives $(\ref{B7.8.2})$ with $C>80$. We have proved Proposition $\ref{1441}$ with $C>347$.
\end{proof}

Next, we combine Proposition \ref{1441} and {Theorem} \ref{com1} to prove
\begin{proposition}\label{1536}
 Given $0<\delta_{1}, \delta_{2}<1$, there exist $\epsilon=\epsilon\left(\delta_{1}, \delta_{2}\right)>0$ and $\tau=\tau\left(\delta_{1}, \delta_{2}\right)>0$ such that the following holds. Let $q, \mu$ be given as in Proposition \ref{1441}. Assume
 \begin{itemize}
     \item[(i)] The $2$-norm of $\hat{\mu}$ is large: \begin{equation}
\label{B7.42}\sum_{\xi \in \hat{\cO/\cP^n}}|\hat{\mu}(\xi)|^{2}>q^{\delta_{1}}.
\end{equation}
 \item[(ii)] 
For all congruential rings $\subringtwo< \mathcal O/{\mathcal P^\startingpower}$, $\cosetpara, \subringfactor\in \mathcal O/{\mathcal P}^\startingpower$ such that $$ [\mathcal O/\mathcal P^\startingpower: \subringfactor \subringtwo ]> q^{\epsilon},$$
we have  
\begin{equation}
\label{B7.43}
\mu( \cosetpara+\subringfactor \subringtwo ) <  [\mathcal O/\mathcal P^\startingpower: \subringfactor\subringtwo ]^{-\delta_2}.
\end{equation}
 \end{itemize}
Then
\begin{equation}
\label{B7.44}\sum\limits_{ y \in \cO/\cP^n}\sum_{\xi \in \hat{\cO/\cP^n}}|\hat{\mu}(\xi)|^{2}|\hat{\mu}(y \xi)|^{2} \mu(y)<q^{-\tau} \sum_{\xi \in \hat{\cO/\cP^n}}|\hat{\mu}(\xi)|^{2}.
\end{equation}
\end{proposition}
\begin{proof}
Take $\tau=\tau(\delta_1, \delta_2)>0$ sufficiently small (specified at \eqref{0713}) and assume $(\ref{B7.44})$ fails. Take
\begin{equation}
\label{B7.45}
\epsilon=2\tau / \delta_{2}>\tau,
\end{equation}
so that both $(\ref{B7.6.1})$ and $(\ref{B7.6.2})$ fail. Then we get $\bar{S}$ by Proposition \ref{1441} which satisfies $(\ref{B7.6})$, $(\ref{B7.7})$, $(\ref{B7.38.1})$, $(\ref{B7.17.1})$ and $(\ref{B7.8.2})$.
From $(\ref{B7.38.1}), (\ref{B7.6})$ and $(\ref{B7.42})$,
\begin{equation}
\label{B7.46}
|\mathfrak{p}^{\smallerpower^\prime}\bar{S}|=|\bar{S}|<\frac{q^{1+2 \epsilon}}{\sum\limits_{\xi \in \hat{\cP/\cP^n}}|\hat{\mu}(\xi)|^{2}} < q^{1-\delta_{1}+\frac{4\tau}{\delta_{2}}}<q^{1-\frac{1}{2} \delta_{1}},
\end{equation}
since
\begin{equation}
\label{B7.47}
\tau<\frac{1}{8} \delta_{1} \delta_{2} .
\end{equation}
By $(\ref{B7.43})$, for all congruential rings $\subringtwo< \mathcal O/{\mathcal P^\startingpower}$, $\cosetpara, \subringfactor\in \mathcal O/{\mathcal P}^\startingpower$ such that $ [\mathcal O/\mathcal P^\startingpower: \subringfactor \subringtwo ]> q^{\epsilon},$
\begin{align}
\mu* \mu_-( \cosetpara+\subringfactor \subringtwo)&=\sum_{x\in \cO/\cP^n}\mu_-(x)\mu(a-x+\subringfactor \subringtwo )\nonumber \\
&< \sum_{x\in R} \mu_-(x) [\mathcal O/\mathcal P^\startingpower: \subringfactor \subringtwo ]^{-\delta_2} \nonumber\\
&=  [\mathcal O/\mathcal P^\startingpower: \subringfactor \subringtwo]^{-\delta_2},\label{B7.43.1}
\end{align}
which shows that $\mu* \mu_-$ satisfies the same non-concentration property as $(\ref{B7.43})$. In particular, setting $\subringtwo=\mathcal O/\mathcal P^\startingpower, \subringfactor=0$, we obtain
\begin{equation}
\mu* \mu_-(a)<q^{-\delta_2} \text{ for any }a\in\cO/\cP^n.   
\end{equation}
On the other hand, by $(\ref{B7.8.2})$,
\begin{equation}
\mu* \mu_-(\mathfrak{p}^{\smallerpower^\prime}\bar{S})>q^{-C\epsilon},
\end{equation}
so
\begin{equation}
\label{N02}
|\mathfrak{p}^{\smallerpower^\prime}\bar{S}|>q^{-C\epsilon+\delta_2}>q^{\delta_2/2}
\end{equation}
provided that
\begin{equation}
\epsilon<\frac{\delta_2}{2C}.
\end{equation}
Then $(\ref{N02})$ and $(\ref{B7.46})$ gives
$$q^{\delta_2/2}<|\fp^{m'}\bar{S}|<q^{1-\delta_1/2}.$$

 Take $\bar{\epsilon}=\epsilon\left(\min(\delta_{1} / 2, \delta_2/2), \delta_{2} / 2\right)$ in Theorem $\ref{sumproduct}$. We then compute for congruential rings $\subringtwo< \mathcal O/{\mathcal P^\startingpower}$, $\cosetpara, \subringfactor\in \mathcal O/{\mathcal P}^\startingpower$ such that $ [\mathcal O/\mathcal P^\startingpower: \subringfactor\subringtwo]> q^{\bar{\epsilon}},$
\begin{align}
q^{-2\epsilon}\phi(0)|\mathfrak{p}^{\smallerpower^\prime}\bar{S}\cap (\cosetpara+\subringfactor\subringtwo)|&\stackrel{(\ref{B7.17.1})}{<}\phi\left(\mathfrak{p}^{m^\prime}\bar{S}\cap (\cosetpara+\subringfactor\subringtwo)\right)\nonumber\\
&\stackrel{(\ref{B7.10.0}),(\ref{B7.43.1})}{<}q[\mathcal O/\mathcal P^\startingpower: \subringfactor\subringtwo ]^{-\delta_2}\nonumber\\
&\stackrel{(\ref{B7.8.2})}{<}q^{C\epsilon}\phi(\mathfrak{p}^{\smallerpower^\prime}\bar{S})[\mathcal O/\mathcal P^\startingpower: \subringfactor\subringtwo ]^{-\delta_2}\nonumber\\
&\stackrel{(\ref{B7.11})}{<}q^{C\epsilon}\phi(0)|\mathfrak{p}^{\smallerpower^\prime}\bar{S}|[\mathcal O/\mathcal P^\startingpower: \subringfactor\subringtwo ]^{-\delta_2},\nonumber
\end{align}
which implies
\begin{align}
|\mathfrak{p}^{\smallerpower^\prime}\bar{S}\cap (\cosetpara+\subringfactor\subringtwo)|&<q^{(C+2)\epsilon}|\mathfrak{p}^{\smallerpower^\prime}\bar{S}|[\mathcal O/\mathcal P^\startingpower: \subringfactor\subringtwo]^{-\delta_2}\nonumber\\
&=(q^{\bar{\epsilon}})^{(C+2)\frac{\epsilon}{\bar{\epsilon}}}|\mathfrak{p}^{\smallerpower^\prime}\bar{S}|[\mathcal O/\mathcal P^\startingpower: \subringfactor\subringtwo ]^{-\delta_2}\nonumber\\
&<[\mathcal O/\mathcal P^\startingpower: \subringfactor\subringtwo ]^{-\delta_2+(C+2)\frac{\epsilon}{\bar{\epsilon}}}|\mathfrak{p}^{\smallerpower^\prime}\bar{S}|\nonumber\\
&< [\mathcal O/\mathcal P^\startingpower: \subringfactor\subringtwo]^{-\frac{1}{2}\delta_2}|\mathfrak{p}^{\smallerpower^\prime}\bar{S}|,\label{B7.51.1}
\end{align}
given that
\begin{equation}
\label{B7.52.1}
\epsilon<\frac{\bar{\epsilon} \delta_{2}}{2(C+2)}.
\end{equation}
Then by $(\ref{N02}), (\ref{B7.46}), (\ref{B7.51.1})$ we can apply Theorem $\ref{sumproduct}$ to $\mathfrak{p}^{\smallerpower^\prime}\bar{S}$ with $(\delta_1, \delta_2)=\left(\min(\delta_{1} / 2, \delta_2/2), \delta_{2} / 2\right)$ and get
\begin{equation}
|\mathfrak{p}^{\smallerpower^\prime}\bar{S}+\mathfrak{p}^{\smallerpower^\prime}\bar{S}|+|\mathfrak{p}^{\smallerpower^\prime}\bar{S}\cdot \mathfrak{p}^{\smallerpower^\prime}\bar{S}|>|\mathfrak{p}^{\smallerpower^\prime}\bar{S}|^{1+\delta_3}
\end{equation}
with $\delta_3=\delta_3(\min(\delta_{1} / 2, \delta_2/2), \delta_{2} / 2)$. It's clear that
\begin{equation*}
\mathfrak{p}^{\smallerpower^\prime}\bar{S}+\mathfrak{p}^{\smallerpower^\prime}\bar{S}=\mathfrak{p}^{\smallerpower^\prime}(\bar{S}+\bar{S}),
\end{equation*}
\begin{equation*}
\mathfrak{p}^{\smallerpower^\prime}\bar{S}\cdot\mathfrak{p}^{\smallerpower^\prime}\bar{S}=\mathfrak{p}^{2\smallerpower^\prime}(\bar{S}\cdot\bar{S}),
\end{equation*}
and
\begin{equation*}
|\mathfrak{p}^{\smallerpower}A|\leq |A|.
\end{equation*}
It follows that
\begin{equation}
\label{B7.7.1}
|\bar{S}+\bar{S}|+|\bar{S} \cdot \bar{S}| >|\bar{S}|^{1+\delta_3}.
\end{equation}
If
\begin{equation}
\label{B7.53.1}
\epsilon<\frac{\delta_2\delta_3}{2C},
\end{equation}
then
\begin{equation}
q^{C\epsilon}<q^{\frac{1}{2}\delta_2\delta_3}<|\bar{S}|^{\delta_3}
\end{equation}
by $(\ref{N02})$, so $(\ref{B7.7.1})$ contradicts $(\ref{B7.7})$. \par
Summarize all assumptions on parameters from $(\ref{B7.45})$, $(\ref{B7.47})$, $(\ref{B7.52.1})$, $(\ref{B7.53.1})$, it suffices to take
\begin{align}\label{0713}
\tau<\min\left(\frac{1}{8}\delta_{1} \delta_{2}, \frac{1}{4(C+2)}\bar{\epsilon} \delta_{2}^{2}, \frac{1}{4C}\delta_{2}^2 \delta_{3}\right)
\end{align}
and
\[\epsilon=2\tau/\delta_2.\]
This proves Proposition \ref{1536}.

\end{proof}

\newpage

\section{Proof of Theorem \ref{exp1}}
Now we are ready to prove Theorem \ref{exp1}.  We first prove Theorem \ref{exp1} in the special case that all $\mu_i$ are supported on $(\cO/\cP^n)^*$, and $\mu_i=\mu_0$ for some $\mu_0$.  Then by using a elementary combinatorial identity, we prove Theorem \ref{exp1} for all $(\cO/\cP^n)^*$ supported on  $(\cO/\cP^n)^*$, without assuming all $\mu_i$ are the same.  After that, we prove Theorem \ref{exp1} in full generality.  \par
\subsection{The case $\mu_i=\mu_0$ and $\normalfont\text{Supp}(\mu_0)\subset (\cO/\cP^n)^*$}
Given $\mu_i=\mu_0$ satisfying \eqref{nonconcentration1}, write $\nu_k=\mu^{\otimes k}$.  It is straightforward that 
\begin{align}\label{B7.55}
\sum_{y}\nu_k(y)\hat{\nu_k}(y\xi)=\hat{\nu_{2k}}(\xi).
\end{align} 

Define for $k \in \mathbb{Z}_{+}, \kappa>0$ the set
$$
\Omega_{k, \kappa}=\left\{\xi \in R|| \hat{\nu}_{k}(\xi) \mid>q^{-\kappa}\right\} .
$$
If $\xi \in \Omega_{2 k, \kappa}, (\ref{B7.55})$ implies
$$
\sum_{x}\left|\sum_{y}\xi(xy) \nu_{k}(y)\right| \nu_{k}(x)=\sum_{x}\left|\hat{\nu}_{k}(x\xi)\right| \nu_{k}(x)\geq \left|\hat{\nu}_{2k}(\xi)\right|>q^{-\kappa},
$$
 Hence, for $r \in{Z}_{+}$,
\begin{equation}
\label{B7.56} \sum_{x}\left|\sum_{y}\xi(xy) \nu_{k}(y)\right|^{2 r} \nu_{k}(x)>q^{-2 \kappa r} .
\end{equation}
The left side of $(\ref{B7.56})$ equals
\begin{align*}
&\ \ \sum_{x}\left(\sum_{y}\xi(xy) \nu_{k}(y)\sum_{y} \overline{\xi(xy)} \nu_{k}(y)\right)^{r}\nu_{k}(x)\\
&=\sum_{x}\sum_{y_{1}, \ldots, y_{2r} }\xi((y_{1}-y_{2}+\cdots+y_{2r-1}-y_{2r})x)\nu_{k}\left(y_{1}\right) \cdots \nu_{k}\left(y_{2 r}\right)\nu_{k}(x)\\
&=\sum_{y_{1}, \ldots, y_{2r} } \hat{\nu_{k}}\left(\left(y_{1}-y_{2}+\cdots-y_{2r}\right) \xi\right) \nu_{k}\left(y_{1}\right) \cdots \nu_{k}\left(y_{2 r}\right)\\
&=\sum_{z} \hat{\nu}_{k}(z \xi)\left(\nu_{k}*\nu_{k}^{-}\right)^{(r)}(z)\\
&\leq \left(\sum_{z}\left|\hat{\nu}_{k}(z \xi)\right|^{4 r}\left(\nu_{k} * \nu_{k}^{-}\right)^{(r)}(z)\right)^{\frac{1}{4r}},
\end{align*}
where in last step we applied H\"older's inequality. We denote here $\nu_k^{-}(x)=\nu_k(-x)$ and by $\nu_k^{(r)}$ the $r$-fold additive convolution of $\nu$.
Then, if $\xi \in \Omega_{2 k, \kappa}$,
$$
\sum_{z}\left|\hat{\nu}_{k}(z \xi)\right|^{4 r}\left(\nu_{k} * \nu_{k}^{-}\right)^{(r)}(z)>q^{-8 \kappa r^{2}}.
$$
Hence
\begin{equation}
\label{B7.57}
\sum_{z, \xi}\left|\hat{\nu}_{2 k}(\xi)\right|^{4 r}\left|\hat{\nu}_{k}(z \xi)\right|^{4 r}\left(\nu_{k} * \nu_{k}^{-}\right)^{(r)}(z)>q^{-12 \kappa r^{2}} \mid \Omega_{2 k, \kappa }\mid .
\end{equation}
We apply Proposition \ref{1536} with
\begin{equation}
\label{mu1}
\mu=\frac{1}{2}\left[\left(\nu_{k} * \nu_{k}^{-}\right)^{(r)}+\left(\nu_{2k} * \nu_{2k}^{-}\right)^{(r)}\right],
\end{equation}
so that
\begin{align}
\hat{\mu}(\xi)&=\frac{1}{2}\left[\left(\hat{\nu}_{k}(\xi)\hat{{\nu}_{k}^{-}}(\xi)\right)^{r}+\left(\hat{\nu}_{2k}(\xi)\hat{{\nu}_{2k}^{-}}(\xi)\right)^{r}\right]\nonumber\\
&=\frac{1}{2}\left(\left|\hat{\nu}_{k}(\xi)\right|^{2r}+\left|\hat{\nu}_{2k}(\xi)\right|^{2r}\right).\label{muhat1}
\end{align}

We get from $\mathrm{Supp}(\nu_{k}) \subset (\cO/\cP^n)^{*}$ and from \eqref{B7.55} that
\begin{align*}
\sum_{\xi\in\hat{\cO/\cP^n}}|\hat{\nu}_{2k}(\xi)|^{4r}&\leq \sum_{{\xi\in\hat{\cO/\cP^n}}}\sum_{x \in \cO/\cP^n}|\hat{\nu}_{k}(x \xi)\nu_{k}(x)|^{4r}\\
&=\sum_{{\xi\in\hat{\cO/\cP^n}}}\sum_{x \in \cO/\cP^n} |\hat{\nu}_{k}(x \xi)|^{4r}\nu_{k}(x)^{4r}\\
&=\sum_{{\xi\in\hat{\cO/\cP^n}}}\sum_{x \in \cO/\cP^n} |\hat{\nu}_{k}(\xi)|^{4r}\nu_{k}({x})^{4r}\\
&\leq\sum_{{\xi\in\hat{\cO/\cP^n}}} |\hat{\nu}_{k}(\xi)|^{4r}\left(\sum_{x \in \cO/\cP^n}\nu_{k}({x})\right)^{4r}\\
&=\sum_{\xi} |\hat{\nu}_{k}(\xi)|^{4r}:= q^{1-\delta_{k, r}}.
\end{align*}

Therefore, from \eqref{muhat1}, we have
$$
\frac{1}{4}q^{1-\delta_{k,r}}\leq \sum_{\xi}|\hat{\mu}(\xi)|^2\leq q^{1-\delta_{k,r}}.
$$
It is clear that the numbers $\delta_{k, r}$ are increasing in both $k$ and $r$. \par
For two measures $\omega_1,\omega_2$ on $\cO/\cP^n$ and for any $S \subset \cO/\cP^n$,
\[\omega_1* \omega_2(S):=\sum_{y\in \cO/\cP^n}\omega_1(y)\omega_2(S-y)\leq \max_{y\in \cO/\cP^n}{\omega_2}(y+S),\]
and if $\text{Supp}(\omega_1), \text{Supp}(\omega_2)\subset (\cO/\cP^n)^*$, then 
\[\omega_1\otimes \omega_2(S):=\sum_{y\in (\cO/\cP^n)^*}\omega_1(y)\omega_2(y^{-1}S)\leq \max_{y\in (\cO/\cP^n)^*}{\omega_2}(yS).\]

From the defination of $\nu_k$, 
\[{\nu}_{k} * {\nu}_{k}^{-}(S)\leq\max_{x\in \cO/\cP^n}{\nu_k}(x+S)\leq \max_{x\in \cO/\cP^n, y\in (\cO/\cP^n)^*}\mu_0(x+yS) \]
and
\[{\nu}_{2k} * {\nu}_{2k}^{-}(S)\leq\max_{x\in \cO/\cP^n}{\nu_{2k}}(x+S)\leq  \max_{x\in \cO/\cP^n, y\in (\cO/\cP^n)^*}\mu_0(x+yS).\]

So we have
\begin{equation}
\label{B7.58} \mu(S) \leq \frac{1}{2}[\max_{x\in \cO/\cP^n}{\nu_k}(x+S)+\max_{x\in \cO/\cP^n}{\nu_{2k}}(x+S)]\leq \max_{x\in \cO/\cP^n, y\in (\cO/\cP^n)^*}\mu_0(x+yS).
\end{equation}
Therefore, for all congruential rings $\subringthree< \mathcal O/\mathcal P^\startingpower$, $\cosetpara, \subringfactor\in \mathcal O/{\mathcal P}^\startingpower$ such that $$ [\mathcal O/\mathcal P^m: \subringfactor \subringthree]> q^{\epsilon},$$
we have  
\[\mu(\cosetpara+\subringfactor \subringthree ) \leq \max_{x\in \cO/\cP^n, y\in (\cO/\cP^n)^*}\mu_0(x+y(\cosetpara+\subringfactor\subringthree))\]
\[=\max_{x\in \cO/\cP^n, y\in (\cO/\cP^n)^*}\mu_0((x+y\cosetpara)+y\subringfactor \subringthree)
<  [\mathcal O/\mathcal P^m: \subringfactor \subringthree ]^{-\gamma}\]
by assumption \eqref{nonconcentration1} on $\mu_0$. Hence condition $(\ref{B7.43})$ in Proposition \ref{1536} is valid with $\delta_{2}=\gamma.$ \par 

Take
\begin{equation}
\label{B7.74}
\delta_{1}=\gamma / 2.
\end{equation}
 Assume
\begin{equation}
\label{B7.59} 1-\delta_{k, r}>\delta_{1},
\end{equation}
so that $(\ref{B7.42})$ holds for $\mu$ defined in \eqref{mu1}. Then $(\ref{B7.44})$ holds, with $\tau_0=\tau\left(\delta_{1}, \gamma\right)\in (0,1)$. \par
Recall \eqref{mu1} and \eqref{muhat1}.  We have

\begin{align}\nonumber
&\sum_{y\in \cO/\cP^n}\sum_{\xi \in \hat{\cO/\cP^n}}\left|\hat{\nu}_{2 k}(\xi)\right|^{4 r}\left|\hat{\nu}_{k}(y \xi)\right|^{4 r}\left(\nu_{k}*\nu_{k}^{-}\right)^{(r)}(y)\\\leq &\sum_{y\in \cO/\cP^n}\sum_{\xi \in \hat{\cO/\cP^n}}(4|\hat{\mu}(\xi)|^{2})(4|\hat{\mu}(y \xi)|^{2}) (2\mu(y))\nonumber\\
\nonumber<&32q^{-\tau_0} \sum_{\xi \in \hat{\cO/\cP^n}}|\hat{\mu}(\xi)|^{2} \\ 
 <&q^{-\tau_0+} \sum_{\xi \in \hat{\cO/\cP^n}}\left|\hat{\nu}_{k}(\xi)\right|^{4 r} \label{1607}\\
=&q^{1-\delta_{k, r}-\tau_0+}
\label{B7.60}
\end{align}
for sufficiently large $q$, where at \eqref{1607} we again used Bourgain's notation.  The inequality means $32q^{-\tau_0} \sum_{\xi \in \hat{\cO/\cP^n}}|\hat{\mu}(\xi)|^{2}<q^{-\tau_0+\varepsilon} \sum_{\xi \in \hat{\cO/\cP^n}}\left|\hat{\nu}_{k}(\xi)\right|^{4 r}$ for arbitrarily small $\varepsilon>0$ when $q$ is sufficiently large.  \par
Combining \eqref{B7.60} and \eqref{B7.57}, we obtain
\begin{equation}
\label{B7.61}
\left|\Omega_{2 k, \kappa}\right|<q^{1-\delta_{k, r}-\tau_0+12 \kappa r^{2}+} .
\end{equation}
Take
\begin{equation}
\label{B7.62}
\kappa=\tau_0 /\left(40 r^{2}\right)<1/40<1/4
\end{equation}
and let
\begin{equation}
\label{B7.63}
\bar{r}=[1 / \kappa]\geq 4.
\end{equation}
We get
\[\bar{r}+1>1 / \kappa\]
and
\[\bar{r}\kappa>1-\kappa>3/4.\]
From $(\ref{B7.61})$,
\begin{equation}
\label{B7.64}\left|\Omega_{2 k, \kappa}\right|<q^{1-\delta_{k, r}-\frac{1}{2} \tau_0}
\end{equation}
and
\[q^{1-\delta_{2 k, \bar{r}}}=\sum_{\xi}\left|\hat{\nu}_{2 k}(\xi)\right|^{4 \bar{r}}
=\sum_{\xi\in \Omega_{2 k, \kappa}}\left|\hat{\nu}_{2 k}(\xi)\right|^{4 \bar{r}}+\sum_{\xi\notin \Omega_{2 k, \kappa}}\left|\hat{\nu}_{2 k}(\xi)\right|^{4 \bar{r}}\]
\[\leq\left|\Omega_{2 k, \kappa}\right|+q \cdot q^{-4 \tilde{r} \kappa}<q^{1-\delta_{k, r}-\frac{\tau_0}{2}}+q^{-2}<q^{1-\delta_{k, r}-\frac{\tau_0}{3}},\]
so
\begin{equation}
\label{B7.65} 
\delta_{2 k, \bar{r}} >\delta_{k, r}+\tau_0 / 3.
\end{equation}
By $(\ref{B7.62})$, $(\ref{B7.63})$,
\begin{equation}
\label{B7.66}
\bar{r}<1 / \kappa=40 r^{2} \tau_0^{-1}.
\end{equation}
Starting from $k_0=1, r_0=1$, we perform an iteration of the preceding until \eqref{B7.59} is violated. Thus, according to $(\ref{B7.65})$, $(\ref{B7.66})$,
\begin{align}
k_s&=2^s,\\
\delta_{2^{s}, r_{s}} &>\delta_{2^{s-1}, r_{s-1}}+\tau_0 / 3, \label{B7.67}\\
r_{s} &<40 \tau_0^{-1} r_{s-1}^{2} \label{B7.68}
\end{align}
and the iteration terminates at step $s^{\prime} \geq 0$ when
\begin{equation}
\label{B7.69}
\delta_{2^{s^{\prime}}, r_{s^{\prime}}}>1-\delta_{1} .
\end{equation}
It's guaranteed that
\begin{equation}
\label{B7.70}
s^{\prime}<3 / \tau_0.
\end{equation}
Note that for $s\geq 1$
\[\frac{\log r_s}{2^s}<\frac{\log(40 \tau_0^{-1})}{2^s}+\frac{\log r_{s-1}}{2^{s-1}},\]
so
\begin{equation}
\frac{\log r_s}{2^s}<\log(40 \tau_0^{-1})
\end{equation}
and so 
\begin{equation}
\label{B7.70.1}
r_s<(40 \tau_0^{-1})^{2^s}.
\end{equation}

Denoting $k^{\prime}=2^{s^{\prime}}, r^{\prime}=r_{s^{\prime}}$, we obtain
\begin{equation}
\label{B7.71}
\sum_{\xi}\left|\widehat{\nu_{k^{\prime}}}(\xi)\right|^{4 r^{\prime}}<q^{\delta_{1}}
\end{equation}
It follows by $(\ref{B7.55})$ again that for all $\xi_{0} \in \hat{\cO/\cP^n}$ and $k>k^{\prime}$,
$$
\hat{\nu}_{k}\left(\xi_{0}\right)=\sum_{x \in \cO/\cP^n} \widehat{\nu_{k-1}}\left(x \xi_{0}\right)\mu_1(x).
$$
By H\"older's inequality, and that $\mu_1(x)<q^{-\gamma}$ for all $x \in\cO/\cP^n$, 
\begin{align}
\nonumber
\left|\hat{\nu}_{k}\left(\xi_{0}\right)\right|\leq &\left(\sum_{x\in\cO/\cP^n}\left|\hat{\nu}_{k-1}\left(x \xi_{0}\right)\right|^{4 r^{\prime}}\right)^{\frac{1}{4r^{\prime}}}\left(\sum_{\cO/\cP^n}\mu_1(x)^{\frac{4r^{\prime}}{4r^{\prime}-1}}\right)^{\frac{4r^{\prime}-1}{4r^{\prime}}} \\
< & q^{\frac{\delta_{1}}{4r^{\prime}}}q^{-\frac {\gamma}{4r^{\prime}}}\label{B7.72},
\end{align}
which implies that
\begin{equation}
\label{B7.75}
\max _{\xi \in (\cO/\cP^n)^{*}}\left|\hat{\nu}_{k}(\xi)\right|<q^{-\gamma / 8 r^{\prime}}.
\end{equation}

To summarize, we can take
\begin{align}
\nonumber &k=[8^{\frac{1}{\tau_0}}], \\
\nonumber &r^{\prime}=(40 \tau_0^{-1})^{8^{\frac{1}{\tau_0}}},\\
&\tau=\gamma / 8 r^{\prime}.
\end{align}
\\
\\
This proves Theorem \ref{exp1} assuming $\mu_i=\mu_0$ for all $i$ and $\mathrm{supp}\ \mu_0 \subset (\cO/\cP^n)^{*}$.
\\
\\
\subsection{The case $\normalfont\text{Supp}(\mu_i)\subset (\cO/\cP^n)^*$ }\label{1124}  Consider $\nu_k=\mu_1\otimes \mu_2  \otimes \cdots \otimes \mu_k$. By Lemma $\ref{polyiden}$, we know 
\[\nu_k=\frac{1}{k!}\sum\limits_{i=0}^{k-1}(-1)^iP_{k-i},\]
where
\[P_t=\sum\limits_{1\leq i_1<i_2<\cdots <i_t\leq k} \left(\sum\limits_{j=1}^t\mu_{i_j}\right)^{\otimes k}.\]
We know for any $1\leq i_1<i_2<\cdots <i_t\leq k$, $\zeta=\zeta(i_1, i_2, \cdots , i_t):=\frac{1}{t}\left(\sum\limits_{j=1}^t\mu_{i_j}\right)$ is a probability measure on $R$ since each $\mu_{i_j}$ is. Moreover, $\zeta$ satisfies {\eqref{nonconcentration1}} and $\text{Supp}(\zeta) \subset (\cO/\cP^n)^*$ by properties of $\mu_{i_j}$. By previous case we know there exist $k=k(\gamma), \tau=\tau(\gamma)$ such that
\[\max _{\xi \in R^{*}}\left|\hat{\zeta}_{k}(\xi)\right|<q^{-\tau},\]
where $\zeta_k=\zeta^{\otimes k}$. Therefore, for $\xi \in (\cO/\cP^n)^*$,
\[\left|\hat{P_t}(\xi)\right|< \sum\limits_{1\leq i_1<i_2<\cdots <i_t\leq k} t^k q^{-\tau} \leq k^k q^{-\tau}\binom{k}{t}\leq k^k q^{-\tau}\binom{k}{[\frac{k}{2}]}\]
and thus
\[\left|\hat{\nu_k}(\xi)\right|< \frac{k}{k!}k^k q^{-\tau}\binom{k}{[\frac{k}{2}]}\leq k^{k+1} q^{-\tau} \leq q^{-\tau/2}\]
for sufficiently large $q$. This proves Theorem \ref{exp1} for $\mathrm{Supp}(\mu_{i}) \subset (\cO/\cP^n)^*$.
\\
\\
\subsection{The general case}
Let {$k_0(\gamma), \epsilon_0(\gamma)$ and $\tau_0(\gamma)$} be the functions for $k, \epsilon, \tau$ obtained in Section \ref{1124}, for the special case all $\mu_i$ are supported on invertible elements. Now for the general case, we take 

$$k=k_0(\frac{\gamma}{2}).$$

Let $\mu_1, \cdots, \mu_k$ on $\cO/\cP^n$ be general measures on $\cO/\cP^n$ satisfying \eqref{nonconcentration1}, and let $\chi$ be a general primitive character on $\cO/\cP^n$.  For each $1\leq i\leq k$, decompose 
$$\mu_{k}=\sum_{0\leq v\leq n}\mu_{i, v}, $$
where $\mu_{i,  v}$ is the restriction of $\mu$ to $R_v$, recalling \eqref{1129}.  

It thus suffices to establish for a proper choice of $\tau$, 
\begin{equation}
\label{B7.79}
\left|\sum_{x_i\in R_{v_{i}}, 1\leq i\leq k} \mu_1(x_1)\cdots\mu_k(x_k)\chi(x_1\cdots x_k)\right|<q^{-2\tau}
\end{equation}
for every $0\leq v_1, \cdots, v_k\leq n$, so that the total sum of such is bounded by $q^{-\tau}$.  \par
We assume
\begin{align}\label{1142}
\tau<\epsilon \gamma.
\end{align}

It suffices to assume each $v_i\leq 2\epsilon n$; {otherwise, $|\mu_i|<q^{-2\tau}$ by non-concentration.} Furthermore, we can assume 
\begin{equation}
\label{B7.80} |\mu_{i,v_i}|>q^{-\sigma}
\end{equation}
for all $i$. Here $\sigma=\sigma(\gamma)>0$ is a small constant and will be given at \eqref{1449}. In particular, $2\tau<\sigma$. Indeed, If 
\[\mu_i\left(R_{v_{i}}\right)\leq q^{-\sigma}<q^{-2\tau}\]
for some $i$, then the trivial bound of $(\ref{B7.79})$ gives
\[
\quad \left|\sum_{x_{1} \in R_{v_{1}}, \ldots, x_{k} \in R_{v_{k}}} \mu_1(x_1)\cdots\mu_k(x_k)\chi(x_1\cdots x_k)\right|\leq q^{-\sigma}<q^{-2\tau}.
\]

 Define
$$
\startingpower^{\prime}=\startingpower-v_{1}-\cdots-v_{k}.
$$
Then, $$\nu_{1}+\cdots+\nu_{k} \leq k \epsilon \startingpower<\startingpower / 2$$ if
\begin{equation}
\label{B7.83} \epsilon<\frac{1 }{2 k},
\end{equation}
so that
\begin{equation}
\label{B7.84}
\startingpower^\prime>\frac{1}{2}\startingpower.
\end{equation}

Let $\mu_{i, v_i}^{*}$ be the pushforward of $\mu_{i, v_i}$ under the map $\pi_{n'}\circ \phi_{i}$, where $\phi_i: R_{v_i}\rightarrow \cO/\cP^{n-\nu_i}$ is defined as 
$$\phi_i(\fp^{\nu_i}x)=x.$$ 
Let $\chi^*$ be the primitive character on $\cO/\cP^{n'}$ given by 
$$\chi^*(x)=\chi (\fp^{n-n'}x).$$ 
Then, 
\begin{align}\label{1000}
\reallywidehat{\mu_{1,v_1}\otimes\cdots\otimes\mu_{k,v_k}}(\chi)=\reallywidehat{\mu_{1,v_1}^*\otimes\cdots\otimes\mu_{k,v_k}^*}(\chi^*)
\end{align}

We want to apply Theorem \ref{exp1} to the right hand side of \eqref{1000} with $n$ replaced by $n'$. For this, we need to verify condition \eqref{nonconcentration1}.

We write $Q=|\mathcal O/\mathcal P^{\startingpower^\prime}|=p^{d_0\startingpower^\prime}>p^{\frac{1}{2}d_0\startingpower}=q^\frac{1}{2}$. Let
\begin{equation}
\label{N03}
\bar{\epsilon}=\epsilon_0(\gamma / 2),
\end{equation}
where $\epsilon_0({\gamma}/{2})$ is the constant obtained from the special case of Theorem \ref{exp1} considered in Section \ref{1124}. \par

{
For any congruential subring $R< \mathcal O/{\mathcal P^{\startingpower^\prime}}$, $\cosetpara, \subringfactor\in \mathcal O/{\mathcal P}^{\startingpower^\prime}$ such that  $[\mathcal O/\mathcal P^{\startingpower^\prime}: \subringfactor R_1 ]> Q^{\bar{\epsilon}},$ the preimage of $a+bR$ under canonical reduction $\cO/\cP^n\mapsto \cO/\cP^{n'}$ is also of the form $a'+b'R'$, where $a',b'\in \cO/\cP^n$ and $R'$ congruential. Moreover,  
$$[\cO/\cP^{n}: b'R']=[\cO/\cP^{n'}: bR]>Q^{\bar{\epsilon}}>q^{\epsilon}, $$
if we take 
$$\epsilon<\bar{\epsilon}/2.$$
}
Then 
\begin{align*}
\mu_{i,v_i}^*(a+bR)=&\mu_{i,v_i}(a'+b'R')
<[\cO/\cP^n: b'R')]^{-\gamma}=[\cO/\cP^{n'}: bR]^{-\gamma}.
\end{align*}

If we let 
\begin{align}\label{1449}
\sigma=\frac{\bar{\epsilon}\gamma}{4},
\end{align}
then
\begin{align}
\frac{\mu_{i,v_i}^*(a+bR)}{|\mu_{i,v_i}^*|}<[\cO/\cP^{n'}: a+bR_1]^{-\frac{\gamma}{2}}.
\end{align}

Hence the special case of Theorem \ref{exp1} considered in Section \ref{1124} applies, and we get under assumption $(\ref{B7.80})$,
\begin{align}
\reallywidehat{\mu_{1,v_1}^*\otimes\cdots\otimes\mu_{k,v_k}^*}(\chi^*)< Q^{-\tau_0(\frac{\gamma}{2})}<q^{-\tau_0/2}.
\label{B7.90}
\end{align}

This gives $(\ref{B7.79})$, if we set

$${
\tau=\min \left(\frac{1}{4} \tau_0\left(\frac{\gamma}{2}\right), \frac{\epsilon_0(\frac{\gamma}{2})\gamma}{4}, \epsilon \gamma\right).}
$$
Recalling assumptions $(\ref{B7.83})$, $(\ref{N03})$, we may take
\begin{equation}
\label{B7.92}
\epsilon=\min \left(\frac{1}{2 k_0\left(\frac{\gamma}{2}\right)}, \frac{1}{2} \epsilon_0\left(\frac{\gamma}{2}\right)\right).
\end{equation}
This proves the general case of Theorem \ref{exp1}.

\newpage

\section{First Reductions of Theorem \ref{mainthm}}
In the rest of the paper we finish the proof of Theorem \ref{mainthm}.  Let $K$ be a number field and $\mathcal O$ the ring of integer of $K$. Let $\mathfrak a$ be an ideal of $\mathcal O$ and $A\subset \mathcal O$ be such that $|\pi_{\fa}(A)|>[\mathcal O:\mathfrak a]^{\delta}$.  Write $\mathfrak a=\prod_{i\in I}\mathcal P_i^{n_i}$ where each $\mathcal P_i$ is prime with ramification index $e_i$, $\mathcal O/\mathcal P_i=\mathbb F_{p_i^{l_i}}$, $q=[\mathcal O:\mathfrak a]=\prod_{i\in I}p_i^{n_il_i}$.  For an ideal $\mathfrak b$ of $\mathcal O$, let $\pi_{\mathfrak b}$ be the canonical projection from $\mathcal O$ to $\mathcal O/\mathfrak b$.  For each $\mathcal P_i$, denote by $\mathcal O_{\mathcal P_i}$ the localization of $\mathcal O$ at $\mathcal P_i$.  We also choose for each $i\in I$ an element $\mathfrak p_i\in\mathcal O$ such that $\mathfrak p_i$ is an uniformizing element in $\mathcal O_{\mathcal P_i}$, and $\langle \mathfrak p_i, \mathcal P_j\rangle=\mathcal O$ for each $j\neq i$, where $\langle \mathfrak p_i, \mathcal P_j\rangle$ is the ring generated by $\fp_i$ and $\cP_j$. \par 
Recall the definition of congruential rings in $\cO/\cP^n$ at \eqref{1329}. For a general ideal $\fa$, we say a unital ring $R<\cO/\fa$ is congruential if $R\cong \prod_i R(\mod \cP_i^{n_i})$ and each $R(\mod \cP_i^{n_i})$ is congruential in $\cO/\cP_i^{n_i}$.
\par
The goal of this section is to reduce Theorem \ref{mainthm} to Theorem \ref{sp1}: 

\begin{thm}\label{sp1} Suppose $d$ is a positive integer.  For all $0<\gamma<1$ there exist $\epsilon=\epsilon(\gamma, d)>0, \tau=\tau(\gamma, d)>0$, $r_1=r_1(\gamma, d)\in\mathbb Z_+$, $r_2=r_2(\gamma, d)\in\mathbb Z_+$, and $N=N(\gamma, d)\in\mathbb Z_+$ such that the following holds:  Let $K$ be a number field with extension degree at most $d$, $\mathcal O$ be the ring of integers of $K$ and $\fa=\prod_i\cP_i^{n_i}$ be an ideal of $\fa$ such that $${\normalfont{gcd}}(|\cO/\cP_i|,|\cO/\cP_j|)=1$$ for $i\neq j$, and either $n_i=1$ for all $i$ or $n_i>N$ for all $i$.  \par Let $\mu_i, 1\leq i\leq {r_1}$ be probability measures on $\mathcal O/\fa$. Assume all $\mu_i$ satisfy the following non-concentration condition: \par
For all congruential subrings $R<\cO/\fa, a, b\in \cO/\fa$, with 
$$[\cO/\fa: bR]>|\cO/\fa|^\epsilon,$$
we have 
\begin{align}\label{nonconcentration}
\mu_i(a+bR )<[\cO/\fa:  bR]^{-\gamma}.
\end{align}
Then, 
$$\sum_{r_2}\text{Supp}(\mu_{1})\cdots\text{Supp}(\mu_{r_1})-\sum_{r_2}\text{Supp}(\mu_{1})\cdots\text{Supp}(\mu_{r_1})\supset \cL_2/\cL_1.$$
for some ideal $\fa<\cL_1<\cL_2$ with $$|\cL_2/\cL_1|>|\cO/\fa|^{\tau}.$$
\end{thm}

\begin{lemma}\label{secondreduction} To prove Conjecture \ref{BGC}, it suffices to assume $\mathfrak a=\prod_{i\in I}\mathcal P_i^{n_i}$ with $N(\mathcal P_i)$ and $N(\mathcal P_j)$ are coprime if $i\neq j$. 
\end{lemma}

\begin{remark}
We prefer to work with $\fa$ such that co-prime factors of $\fa$ have co-prime indices, because we need to deal with subrings $R<\cO/\fa$. Since $\{|\pi_{\cP_i^{n_i}} (R)|=p_i^{d_in_i}\}_{i\in I}$ are mutually co-prime, by Goursat's Theorem,  we have the following simple decomposition for $R$:
\begin{align}
R\cong \prod_{i\in I}\pi_{\cP_i^{n_i}} (R).
\end{align}
The above decomposition may not hold if different $\cP_i$ divides a same $(p)$ for a prime number $p$. 
\end{remark}

\begin{proof}[Proof of Lemma \ref{secondreduction}] We write $\mathfrak a=\prod_{1\leq i\leq w}\mathfrak a_i$, where we have a decomposition of $I$ into $w$ sets 
\begin{align}
I=I_1\sqcup I_2\sqcup \cdots\sqcup I_w,
\end{align} 
and each $\mathfrak a_j=\prod_{i\in I_j}\mathcal P_i^{n_i}$ so that $N(\mathcal P_{i_1})$, $N(\mathcal P_{i_2})$ are coprime if $i_1, i_2\in I_j$, $i_1\neq i_2$.  By the Chinese Remainder Theorem, 
\begin{align} \label{decomposition1}
\mathcal O/\mathfrak a\cong \mathcal O/\mathfrak a_1\times\cdots\times \mathcal O/\mathfrak a_w. 
\end{align}

{Since $[K:\mathbb Q]=d$, we can find such a decomposition with $w\leq d$.  Then for some $\mathfrak a_i$, $\pi_{\mathfrak a_i}(A)>|\pi_{\mathfrak a }(A)|^{\frac{1}{w}}>q^{\delta/d}$. Obviously, we have $|\cO/\mathfrak a|\leq q$. Then the conclusion of Conjecture \ref{BGC} for $A$ and $\mathfrak a_i$ implies the conclusion for $A$ and $\mathfrak a$.}

\end{proof}
\begin{proof}[Proof of Theorem \ref{sp1} $\Rightarrow$ Theorem \ref{mainthm}]
We can thus assume coprime factors of $\mathfrak a$ has coprime norms. Recall $|\pi_{\mathfrak a}(A)|>q^{\delta}$.  We take a \emph{minimal} $a_0+b_0 R_0 \subset \cO/\fa $, $a_0, b_0\in \mathcal O/\mathfrak a$ and $R_0$ a unital subring of $\mathcal O/\mathfrak a$, such that 
\begin{align}\label{2336}
|\pi_{\mathfrak a}(A)\cap (a_0+b_0R_0)|>[\mathcal O/\mathfrak a: b_0R_0]^{-\frac{\delta}{2}}|\pi_{\fa}(A)|,
\end{align}
where minimal means there does not exist $a_1+b_1R_1\subsetneq a_0+b_0R_0$ such that \eqref{2336} holds with $a_0+b_0R_0$ replaced by $a_1+b_1R_1$. 
The triple $a_0,b_0, R_0$ exists because $\cO/\fa$ is finite.  \par

Let $\fa_1=\langle \fa, \pi_\fa^{-1}(b_0)\rangle$, and $\fa=\fa_1\fa_2$.  Then by \eqref{2336}, 
\begin{align}
[\cO:\fa_2]=[\fa_1:\fa]=|b_0\mathcal O/{\mathfrak a}|\geq |\pi_{\mathfrak a}(A)\cap (a_0+b_0R_0)|\geq q^{\frac{\delta}{2}}.
\end{align}
We take a set $A'\subset \mathcal O$, such that $A'\mapsto \pi_{\mathfrak a}(A): x\mapsto a_0+b_0x (\mod \mathfrak a)$ is a bijection. Then $|A'|>q^{\delta/2}$, and it follows from the minimality of $a_0+b_0R_0$ that $\pi_{\fa_2}(A')$ is not concentrated in any affine translates of congruential rings in $\pi_{\fa_2 }(R_0)$, i.e. for any $a_2, b_2\in \cO/\fa_2, R_2<\cO/\fa_2$ such that $a_2+b_2R_2\subset \pi_{\fa_2 }(R_0)$ (which implies $a_2\in \pi_{\fa_2 }(R_0)$ and $b_2R_2\subset \pi_{\fa_2 }(R_0)$ ), we have 
\begin{align} 
|\pi_{\fa_2}(A')\cap(a_2+b_2R_2)|\leq [\pi_{\fa_2 }(R_0):b_2R_2]^{-\frac{\delta}{2}}|\pi_{\fa_2}(A')|.
\end{align}

Write $\fa_2=\prod_{i\in I}\cP^{w_i}, 0\leq w_i\leq n_i$.  We decompose $\fa_2$ as $\fa_2=\fa_s\fa_l$, where
  $$\fa_s=\prod_{i\in I, w_i\leq N_1}\mathcal P_i^{w_i}, $$
and 
$$\fa_l=\prod_{i\in I, w_i> N_1}\mathcal P_i^{w_i},$$
for some integer $N_1=N_1(\delta)$ specified at \eqref{1059}.  \par

\noindent{\bf Case 1}: $[\mathcal O/\fa_s]>[\mathcal O/\fa_2]^{1/2}>q^{\delta/4}$, i.e. small exponent factors have a significant contribution to $\fa_2$. \par
 We can further reduce our analysis to square free modulus as follows:  {by the minimality of $a_0+b_0R_0$, we have 
 $$|\pi_{\fa_s}(A')|\geq |R_0(\mod \fa_s)|^{\frac{\delta}{2}}\geq [\cO:\fa_s]^{\frac{\delta}{2d}} \geq q^{\frac{\delta^2}{8d}}.$$

  }

 Let $\fa_3=\prod_{\cP_i\supset \fa_s}\cP_i$, the square free part of $\fa_s$.  We further take $a_3+b_3R_3 \subset \cO/\fa_s$ minimal among all $a_3,b_3\in \cO/\fa_s$, all $R_3$ unital subrings of $\cO/\fa_s$ containing $\fa_3/\fa_s$, such that 
 \begin{align}\label{135}
 |\pi_{\fa_s}(A')\cap (a_3+b_3R_3)|{\geq} [\cO/\fa_s: b_3R_3]^{-\frac{\delta}{4d}}|\pi_{\fa_s}(A')|. 
 \end{align}
  \par

{ 
Write $\fa_s'=\langle\pi_{\fa_s}^{-1}(b_3), \fa_s\rangle=\prod_{i\in I'} \mathcal P_i^{u_i}$ for some $I'\subset I$, and $\fa_4=\prod_{i\in I'} \mathcal P_i$. \eqref{135} then implies
 $$|\cO/\fa_4|>|\cO/\fa_{3}'|^{\frac{1}{N_1}}>|\cO/\fa_s|^{\frac{\delta}{4dN_1}}>q^{\frac{\delta^2}{16dN_1}}.$$
 }

For each $\cP_i\supset \fa_4$, it follows from Hilbert ramification theorem that there is a field $K_i< K$ with ring of integers $\cO_i< \cO$ such that $R_3(\mod \cP_i)=\cO_i/(\cO_i\cap \cP_i)$.  Since the number of subfields of $K$ is $O_d(1)$, there exists $K'$ such that $\fa_5=\prod_{i: K_i=K'}\cP_i$ satisfies $$|\cO/\fa_5|>|\cO/\fa_4|^{\Omega_d(1)}>q^{\frac{\delta^2\Omega_d(1)}{N_1}}.$$  

Take $A''\subset \cO$ such that $$A'\mapsto \pi_{\fa_s}(A')\cap a_3+b_3R_3: x\mapsto a_3+b_3x (\mod \fa_s)$$ gives a bijection.  Let $\mu_0$ be the uniform probability measure supported on {$A'\subset \cO$}, and let $\tilde\mu$ be the push-forward of $\mu_0$ by $\pi_{\fa_5}$, then by the minimality of $a_3+b_3R_3$, $\tilde\mu$ satisfies the non-concentration assumption of $\mu_i$ in Theorem \ref{sp1} with {$\fa=\fa_5$, $\gamma=\frac{\delta}{4d}$,  and $\epsilon=\epsilon(\frac{\delta}{4d})$.} The conclusion of Theorem \ref{sp1} for $\mu_i=\tilde{\mu}$ easily gives Theorem \ref{mainthm} for the original set $A$.
 
\noindent{\bf Case 2}: $[\mathcal O/\fa_l]>[\mathcal O/\fa_2]^{1/2}>q^{\frac{\delta}{4}}$. \par

Take 
\begin{align}\label{219}
{C>\max\{\frac{8}{\delta}, \frac{2}{\epsilon_1}\}}
\end{align} 
where  $\epsilon_1=\epsilon(\frac{\delta}{4})$ is the implied constant for $\epsilon$ at Theorem \ref{sp1}. 

Apply Theorem \ref{structure} to $R_0$ with $C$ given at \eqref{219}. Let 
\begin{align}\label{1059}
N_1=n(C) 
\end{align}
be the implied {lower bound at \eqref{lowerbd}},
and for each $\cP_i\supset \fa_l$, let $K_i<K$ be the implied subfield for the ring $R_0(\mod \cP_i)$, $\cO_i$ be the ring of integer of $K_i$, and let $a_i$, $b_i$ be the implied integers.  Since there are $O_d(1)$ many subfields of $K$, there exists $K'$ such that, if we set 
$$J=\{j\in I: w_j>N_1, K_j=K'\},$$
then 
$$\fa_6:=\prod_{\substack{j\in J }}\cP_j^{w_j}$$
satisfies $$|\cO/\fa_6|>|\cO/\fa_2|^{\Omega_d(1)}.$$

For each $\cP_j\supset \mathfrak a_6$, let $a_j, b_j$ be the implied integers for $a$ and $b$ from Theorem \ref{structure}. \par
 Let $\fa_7=\prod_{\substack{j\in J }}\cP_j^{b_j}$ and $\fa_8=\prod_{\substack{j\in J }}\cP_j^{a_j}$. Then $\fa_8/\fa_7\subset R_0(\mod \fa_7)$, and
$$[\cO/\fa_8]<|\cO/\fa_7|^{\frac{1}{C}}.$$

Let $\mu_0$ be the uniform probability measure supported on the set $A'\subset \cO$, and let $\tilde{\mu}$ be the pushfoward of $\mu_0$ under the projection $\cO \rightarrow\cO/\fa_7$.  We verify in the following that $\tilde{\mu}$ satisfies the non-concentration assumptions of Theorem \ref{sp1} for $\mu_i$, with $\fa=\fa_7$, $\gamma=\frac{\delta}{4}$, $\epsilon_1=\epsilon(\frac{\delta}{4})$ be the implied constant from Theorem \ref{sp1}, then the conclusion of Theorem \ref{sp1} for $\mu_i=\tilde{\mu}$ easily gives Theorem \ref{mainthm}.  \par
 Let $R<\cO/\fa_7$ be congruential, $a,b\in \cO/\fa_7$ satisfy 
\begin{equation}
[\cO/\fa_7: bR]>|\cO/\fa_7|^{\epsilon_1}.
\end{equation}
Since  $R$ is congruential and $\fb$ is an ideal, $R\cap (\fa_8/\fa_7)=rR'$ for some $r\in \cO/\fa_7$, and $R'$ is a congruential ring.  We have 
\begin{align}
\tilde\mu(a+bR)\leq \sum_{x\in R/ rR'} \tilde\mu(a+bx+brR')
\end{align}
We observe $ brR'\subset \fa_8/\fa_7\subset R_0(\mod \fa_7)$.  Since $\tilde{\mu}(a+bx+brR')\neq 0$ only if $a+bx\in R_0$, in this case, by the minimality of $a_0+b_0R_0$ at \eqref{2336} and by the definition of $\tilde{\mu}$,  we have 
\begin{align}
\nonumber\tilde{\mu}(a+bx+brR')\leq & [R_0/\fa_7: brR']^{-\frac{\delta}{2}}\\
\nonumber\leq& [\cO/\fa_7: br R']^{-\frac{\delta}{2}}[\cO/\fa_7: R_0/\fa_7]^{\frac{\delta}{2}}\\
\leq &[\cO/\fa_7: bR]^{-\frac{\delta}{2}}|\cO/\fa_7|^{\frac{\delta}{2C}}
\end{align}
Therefore, 
\begin{align} \nonumber \tilde\mu(a+bR)\leq & [\cO/\fa_7: bR]^{-\frac{\delta}{2}}|\cO/\fa_7|^{\frac{\delta}{2C}}\cdot [R: rR'] \\
\nonumber \leq  &[\cO/\fa_7: bR]^{-\frac{\delta}{2}}|\cO/\fa_7|^{\frac{\delta}{2C}}\cdot [\cO/\fa_7: \fa_8/\fa_7] \\
\nonumber \leq  &[\cO/\fa_7: bR]^{-\frac{\delta}{2}}|\cO/\fa_7|^{\frac{2}{C}}\\
\leq  &[\cO/\fa_7: bR]^{-\frac{\delta}{4}},
\end{align}
which gives Condition \eqref{nonconcentration} with $\gamma=\frac{\delta}{4}$. 
\end{proof}

\newpage

\section{Amplification}
In this section, we will prove Theorem \ref{sp1} assuming a decomposition of $\cO/\fa$ constructed in Section \ref{decomposition}. We deal with the case that $\fa$ is a product of large powers of prime ideals.  The case that $\fa$ is square free can be dealt in a similar way with minor modifications, for which we omit the details. The calculation in this section is almost identical to that in Section 3, Part II of \cite{Bou08}.  However, to realise the construction we assume, we need to add significantly more twists compared to the one given in \cite{Bou08}, and this is done in Section \ref{decomposition}.

\par 
Given $\gamma$ as in Theorem \ref{sp1}, this construction depends on five parameters $\rho_0$, $\rho_1$, $\rho_2$, $\rho_3$, and $\rho_4$, which all eventually depends on $\gamma$.  Their values, together with $r_1, r_2, \tau, \epsilon$ are specified as follows, for which Theorem \ref{sp1} holds.  We prefix a constant $\kappa<\frac{1}{20}$.  
Take 
\begin{align}
\nonumber&\rho_3=\frac{\gamma}{2},\\
\nonumber &r_1=k(\rho_3) \text{ as given in Theorem \ref{exp1}},\\
\nonumber &\tau=\tau(\rho_3) \text{ as given in Theorem \ref{exp1}},\\
\nonumber & r_2=\lceil\frac{2}{\tau}\rceil,\\
\nonumber &\rho_1=\frac{\kappa}{8r_1r_2},\\
\nonumber &\rho_0=\frac{\rho_1\rho_2}{8C(d)[\log\frac{1}{\rho_2}]^2r_1},\\
\nonumber &\rho_4=\frac{\kappa\rho_0}{4},\\
\nonumber &\rho_2=\epsilon(\rho_3)\rho_4, \text{where $\epsilon(\rho_3)$ is as given in Theorem \ref{exp1}},\\
\label{1748}&\epsilon=\frac{\rho_1\rho_2}{4C(d)r_1[\log \frac{1}{\rho_2}]^2}.
\end{align}
where in \eqref{1748}, $C(d)$ is a positive constant only depends on $d$.

Recall $\fa=\prod_{i\in I}\cP_i^{n_i}$ and $q=|\cO/\fa|$.  Write $\cQ_i=\cP_i^{n_i}$ and $q_i=|\cO/\cQ_i|$.\par

We let 
\begin{align}\label{1104}
 \bar{n}_i=[\rho_4n_i]. 
\end{align}
We will construct a set $J\subset I$, rearranging indices and assuming $J=\{1,2, \cdots, T\}$, such that
\begin{align}
|\cO/\prod_{i\in J} \cP_i^{n_i} |>|\cO/\fa|^{\rho_0},
\end{align}
and  for each $\mu_i, 1\leq i\leq k$, and for each $1\leq s< T$, we decompose $\cO/\fa$ into a union
\begin{align}\label{1339}
\cO/\fa=B_i^{(s)}\sqcup C_i^{(s)}
\end{align}
of a ``good'' set $B_i^{(s)}$ and a ``junk'' set $C_i^{(s)}$, with 
\begin{align}\label{1706}
\mu_i(C_i^{(s)})<\rho_1,
\end{align}
and the set $B_i^{(s)}$ satisfies the condition:  For any $\xi\in \cO/\prod_{1\leq i\leq s}\cQ_i$, for any congruential ring $R<\cO/\cQ_{s+1}$ and $a,b\in \cO/\cQ_{s+1}$ with 
$$[\cO/\cQ_{s+1}:bR]>q_{s+1}^{\rho_2},$$
we have 
\begin{align}\label{1621}
\nonumber&\mu_i\left(\{x\in B_i^{(s)}: \pi_{\cQ_1\cdots \cQ_s}(x)=\xi, \pi_{\cQ_{s+1}}\in a+bR\}\right)\\<&[\cO/\cQ_{s+1}: bR]^{-\rho_3}\mu_i\left(\{x\in B_i^{(s)}: \pi_{\cQ_1\cdots \cQ_s}(x)=\xi\}\right) 
\end{align}

Assuming the decomposition above, we will iteratively apply Theorem \ref{exp1} to perform a multiscale analysis with respect to the filtration $\{\cO/\prod_{1\leq j\leq s}\cQ_j\}_{1\leq s\leq T}$. \par

Take $r_1=r_1(\gamma), r_2=r_2(\gamma)\in\mathbb Z_+$ sufficiently large.  We consider the map $$\psi: \Omega\rightarrow \cO/\fa$$ defined as
\begin{align*}
&\psi((x_1, \cdots, x_{r_1r_2+1}))\rightarrow x_1\cdots x_{r_1}+\cdots+x_{r_1(r_2-1)+1}\cdots x_{r_1r_2}+x_{r_1r_2+1},
\end{align*}
where $$\Omega=\underbrace{\cO/\fa\times\cdots \times \cO/\fa}_{r_1r_2 \text { copies}}\times Z \subset (\cO/\fa)^{r_1r_2+1}$$ and $$Z=\sum_{i\in J}Z_i ,$$ where we choose $Z_i$ to be a set of elements in $\cO/\fa$ such that the projection map $$Z_i\rightarrow \mathcal \cO (\mod \mathcal P_i^{\bar{n}_i})$$ is bijective and $Z_i\equiv 0(\mod \mathcal P_j^{{n}_j})$ for any $j\neq i$.  Write $\bar{q}_j=|\cO/\mathcal P_j^{\bar{n}_j}|$. \par
For $x\in \pi_{Q_1\cdots Q_s}(\Omega)$, let $\psi_s(x)=\pi_{Q_1\cdots Q_s}\circ \psi(\bar{x})$ where $\bar{x}$ is any lift of $x$ in $\Omega$. \par
Let $$\bP=\underbrace{(\mu_1\otimes\cdots \otimes \mu_{r_1})*\cdots * (\mu_1\otimes\cdots\otimes \mu_{r_1})}_{r_2 \text{ copies}}* m_Z,$$
where $m_Z$ is the normalized counting measure on $Z$.  Hence, $\bP$ is a probability measure on $\Omega$. \par
Fix $1\leq s<T$.  Decompose 
$$\Omega=\Omega_0\cup \Omega',$$
where $$\Omega_0=\underbrace{(B_1^{(s)}\times\cdots \times B_{r_1}^{(s)})\times\cdots\times (B_1^{(s)}\times\cdots\times B_{r_1}^{(s)})}_{r_2 \text{ copies}}\times Z$$

For notational convenience, for integers $1\leq i,j\leq r_1r_2$, we let $B_i^{(s)}=B_j^{(s)}$ and $\mu_i=\mu_j$ if $i\equiv j(\mod r_1)$.  \par
 Let $\nu$ (resp., $\nu_s$) be the normalized counting measure on $\cO/\fa$ (resp., $\cO/\prod_{1\leq i\leq s}\cQ_i$). Let $F$ be the density of the image measure $\psi[\mathbb P]$ with respect to $\nu$:
 \begin{align}\label{1629}
 F(\xi)=q\bP(\{x\in\Omega: \psi(x)=\xi\}).
 \end{align}
 We also put $F_s=\bE_s[F]$, where $\{\bE_s\}_{1\leq s\leq T}$ are the conditional expectation operator with respect to the filtration $\{\prod_{1\leq i \leq s}\cQ_i\}_{1\leq s\leq T}$.  Then, for $\xi\in \cO/\prod_{1\leq i \leq s}\cQ_i$, 
 \begin{align}
 F_s(\xi)=q_1\cdots q_s \bP(\{x\in\Omega: \pi_{\prod_{1\leq i \leq s}\cQ_i}\psi(x)=\xi\}). 
 \end{align}

From \eqref{1706}, we have 
\begin{align}\label{1732}
\mathbb P(\Omega')<1-(1-\rho_1)^{r_1r_2}<r_1r_2\rho_1.
\end{align}

Let 

\begin{align}
\mathbb P=\mathbb P\big\vert_{\Omega_0}+\mathbb P\big\vert_{\Omega'}=\mathbb P_0+\mathbb P'
\end{align}
and decompose 
$$F=\frac{d\psi[\mathbb P_0]}{d\nu}+\frac{d\psi [\mathbb P']}{d\nu}=F_0+F'.$$
Then we have
$$F_{s+1}=\mathbb E_{s+1}[F_0]+\mathbb E_{s+1}[F'].$$

From \eqref{1732}, we have

$$\int F'd{\nu} < r_1r_2\rho_1,$$
so that
$$\int \mathbb E_{s+1}[F']<r_1r_2\rho_1.$$
\medskip

Our next aim is to estimate 
\begin{align}
\int\max_{\theta\in \cO/\mathcal Q_{s+1} }\mathbb E_{s+1}[F_0](\xi, \theta)d\nu_s(\xi) .
\end{align}

We have, 
\begin{align}
\nonumber&\mathbb E_{s+1}[F_0](\xi,\theta)\\
\nonumber=& {q_1\cdots q_{s+1}}\bP(\{x\in B_1^{(s)}\times\cdots\times B_{r_1r_2}^{(s)}\times Z: \pi_{Q_1\cdots Q_{s+1}}\psi(x)=(\xi, \theta) \}) \\
\leq & {q_1\cdots q_{s+1}}\sum_{\substack{ y\in \pi_{\mathcal Q_1\cdots \mathcal Q_s}(B_1^{(s)}\times\cdots\times B_{r_1r_2}^{(s)}\times (Z_1\times\cdots\times Z_s)) \\ \psi_s(y)=\xi}} \nonumber\\ &\bP(\left\{x\in B_1^{(s)}\times\cdots\times B_{r_1r_2}^{(s)}\times Z)\Big\vert \substack{ \pi_{Q_1\cdots Q_s}(x_i)= y_i(i\leq r_1r_2+1)\\  \pi_{\mathcal Q_{s+1}}\psi (x)=\theta }  \right\})
\end{align}

Therefore, 
\begin{align}
\nonumber &\max_{\theta\in\cO/\mathcal Q_{s+1}} \mathbb E_{s+1}[F_0](\xi,\theta) \\
\nonumber \leq & {q_1\cdots q_{s+1}}\sum_{\substack{ y\in \pi_{\mathcal Q_1\cdots \mathcal Q_s}(B_1^{(s)}\times\cdots\times B_{r_1r_2}^{(s)}\times (Z_1\times\cdots\times Z_s)  \\ \psi_s(y)=\xi}} \nonumber\\&\max_{\theta\in\cO/\mathcal Q_{s+1} }  \bP(\left\{x\in B_1^{(s)}\times\cdots\times B_{r_1r_2}^{(s)}\times Z)\Big\vert \substack{ \pi_{Q_1\cdots Q_s}(x_i)= y_i(i\leq r_1r_2+1)\\  \pi_{\mathcal Q_{s+1}}\psi (x)=\theta }  \right\})\end{align}

We use Fourier analysis and apply Theorem \ref{exp1} to estimate $\max_{\theta\in\cO/\mathcal Q_{s+1}}|\{\cdots\}|$ above.  \par

Let $\chi$ be a primitive additive character on $\cO/\cQ_{s+1}$.  Then, all additive characters on $\cO/\cQ_{s+1}$ are given by $\{\chi_z: \chi_z(x):=\chi(zx), {z\in \cO/\cQ_{s+1}}\}$.

For a fixed $y$, let $B_i^{(s)}(y_i)=\{x\in B_i^{(s)}: \pi_{\mathcal Q_1\cdots\mathcal Q_s}(x)=y_i\}$. \par
 We have 
\begin{align}
\nonumber& \bP(\left\{x\in B_1^{(s)}\times\cdots\times B_{r_1r_2}^{(s)}\times Z)\Big\vert \substack{ \pi_{Q_1\cdots Q_s}(x_i)= y_i(i\leq r_1r_2+1)\\  \pi_{\mathcal Q_{s+1}}\psi (x)=\theta }  \right\})\\
\nonumber\leq & \frac{1}{\bar{q}_1\cdots\bar{q}_s}\bP(\{x\in B_1^{(s)}(y_1)\times \cdots B_{r_1r_2}^{(s)}(y_{r_1r_2}) \times Z_{s+1}: \pi_{\mathcal Q_{s+1}}\psi(x)=\theta \})  \\
\nonumber=& \frac{1}{\bar{q}_1\cdots\bar{q}_s}\frac{1}{q_{s+1}}\sum_{z\in \cO/ \mathcal Q_{s+1}}\chi_z(-\theta)\sum_{x\in B_1^{(s)}(y_1)\times \cdots B_{r_1r_2}^{(s)}(y_{r_1r_2})\times Z_{s+1}}\bP(x)\chi_z(\psi(x))   \\
\nonumber=& \frac{1}{\bar{q}_1\cdots\bar{q_{s+1}}}\frac{1}{q_{s+1}}\sum_{z\in \cO/ \mathcal Q_{s+1}}\chi_z(-\theta)\left(\sum_{x\in B_1^{(s)}(y_1)\times \cdots B_{r_1}^{(s)}(y_{r_1})}\mu_1(x_1)\cdots\mu_{r_1}(x_{r_1})\chi_z(x_1\cdots x_{r_1}) \right)^{r_2} \\
&\nonumber \times \sum_{w\in Z_{s+1}}\chi_z(w)  \nonumber \\
\leq &  \frac{1}{\bar{q}_1\cdots\bar{q}_{s+1}}\frac{1}{q_{s+1}}\sum_{z\in\cO/ \mathcal Q_{s+1} } \left\vert \sum_{x\in B_1^{(s)}(y_1)\times \cdots B_{r_1}^{(s)}(y_{r_1})}\mu_1(x_1)\cdots\mu_{r_1}(x_{r_1})\chi_z(x_1\cdots x_{r_1}) \right\vert^{r_2} \times \left\vert  \sum_{w\in Z_{s+1}} \chi_z(w)  \right\vert \label{11271}
\end{align}

If $z\in (\cP_i^{n_{s+1}-\bar{n}_{s+1}}-\{0\} )(\mod \cP_{s+1}^{n_{s+1}})$, then $ \sum_{w\in Z_{s+1}} \chi_z(w)=0$, so we can restrict our attention to the terms $z\not\in \cP_{s+1}^{n_{s+1}-\bar{n}_{s+1}}/\cQ_{s+1}$, for which we estimate the $w$-sum trivially and apply Theorem \ref{exp1} to estimate the $x$-sum. \par
We can thus estimate \eqref{11271} as

\begin{align}\nonumber
\eqref{11271}&\nonumber \leq\frac{1}{\bar{q}_1\cdots\bar{q}_{s}}\frac{1}{q_{s+1}}\prod_{i=1}^{r_1r_2}\mu_i(B_i^{(s)}(y_i)) \\&
\nonumber + \frac{1}{\bar{q}_1\cdots\bar{q}_{s+1}}\frac{1}{q_{s+1}}\\&\sum_{\substack{z\in(\cO-\cP_{s+1}^{n_{s+1}-\bar{n}_{s+1}}) / \mathcal Q_{s+1} }}\left\vert \sum_{x\in B_1^{(s)}(y_1)\times \cdots B_{r_1}^{(s)}(y_{r_1})}\mu_1(x_1)\cdots\mu_{r_1}(x_{r_1})\chi_z(x_1\cdots x_{r_1}) \right\vert^{r_2}  \times \left\vert  \sum_{w\in Z_{s+1}} \chi_z(w)  \right\vert \nonumber  \\
\nonumber &\leq \prod_{i=1}^{r_1r_2}\mu_i(B_i^{(s)}(y_i)) \\&\times\left( \frac{1}{\bar{q}_1\cdots\bar{q}_{s}}\frac{1}{q_{s+1}}  + \frac{1}{\bar{q}_1\cdots\bar{q}_{s}}\frac{1}{q_{s+1}}\sum_{0\leq a_{s+1}< n_{s+1}-\bar{n}_{s+1}}\sum_{\substack{z\in\cO/ \mathcal Q_{s+1} \\ \mathfrak p_{s+1}^{a_{s+1}}||z } }|\cO/\cP_i|^{-(n_{s+1}-a_{s+1})\tau r_2} \right)  \nonumber\\ 
\label{0201} &\leq \left( \frac{2}{\bar{q}_1\cdots\bar{q}_{s}q_{s+1}}\right)\prod_{i=1}^{r_1r_2}\mu_i(B_i^{(s)}(y_i)),
\end{align}
if we let 
\begin{align}
r_2\tau\geq 2.
\end{align}

Notice that the estimate \eqref{0201} is independent of $\theta$. Therefore, we obtain the following fundamental estimate:  
\begin{align}\label{2108}
&\int\max_{\theta\in \cO/\mathcal Q_{s+1} }\mathbb E_{s+1}[F_0](\xi, \theta)d\nu_s (\xi)\nonumber \\
\nonumber\leq& \frac{2q_1\cdots q_{s}\cdot q_{s+1}}{\bar{q}_1\cdots\bar{q}_s q_{s+1}}\int \sum_{\substack{ y\in \pi_{\mathcal Q_1\cdots \mathcal Q_s(B^{r_1r_2}\times Z) } \\ \psi_s(y)=\xi}} \prod_{i=1}^{r_1r_2}\mu_i(B_i^{(s)}(y_i))  d\nu_s(\xi) \\
\nonumber= & \frac{2q_1\cdots q_{s}\cdot q_{s+1}}{\bar{q}_1\cdots\bar{q}_s q_{s+1}}\sum_{\xi\in \cO/\cQ_1\cdots \cQ_s} \frac{1}{q_1\cdots q_s}\sum_{\substack{ y\in \pi_{\mathcal Q_1\cdots \mathcal Q_s(B^{r_1r_2}\times Z) } \\ \psi_s(y)=\xi}} \prod_{i=1}^{r_1r_2}\mu_i(B_i^{(s)}(y_i))  \\
\nonumber= & \frac{2}{\bar{q}_1\cdots\bar{q}_s }\sum_{\substack{ y\in \pi_{\mathcal Q_1\cdots \mathcal Q_s(B^{r_1r_2}\times Z) } }} \prod_{i=1}^{r_1r_2}\mu_i(B_i^{(s)}(y_i))  \\
=&2.
\end{align}
\medskip

For a real valued function $f$ on a space $X$, we let 
$$f^+(x)=\begin{cases}f(x),&\text{ if }f(x)\geq 0, \\0 ,&\text{ otherwise.}\end{cases}$$

Recall the definition for $F$ at \eqref{1629}. Our next aim is to estimate $\int F_T\log^+ F_T$, which gives one way to measure the flatness of $\psi(\mathbb P)$.\par

Since for $\xi\in \cO/ \mathcal \mathcal Q_1\cdots \mathcal Q_s$,
\begin{align}\nonumber
&\sum_{\theta\in \cO/\mathcal Q_{s+1}} \frac{F_{s+1}(\xi, \theta)}{q_{s+1}}\log F_{s+1}(\xi,\theta)\\
 = &\sum_{\theta\in \cO/\mathcal Q_{s+1}} \frac{F_{s+1}(\xi, \theta)}{q_{s+1}}\log F_{s}(\xi)+\sum_{\theta\in \cO/\mathcal Q_{s+1}} \frac{F_{s+1}(\xi, \theta)}{q_{s+1}}\log\frac{F_{s+1}(\xi,\theta)}{F_s(\xi)},
\end{align}
we have 
\begin{align}\nonumber
&\sum_{\theta\in \cO/\mathcal Q_{s+1}} \frac{F_{s+1}(\xi, \theta)}{q_{s+1}}\log^+ F_{s+1}(\xi,\theta)\\
\label{0217} \leq  &\sum_{\theta\in \cO/\mathcal Q_{s+1}} \frac{F_{s+1}(\xi, \theta)}{q_{s+1}}\log^+ F_{s}(\xi)+\sum_{\theta\in \cO/\mathcal Q_{s+1}} \frac{F_{s+1}(\xi, \theta)}{q_{s+1}}\log^+\frac{F_{s+1}(\xi,\theta)}{F_s(\xi)}.
\end{align}

Integrating over $\xi$ on both sides of \eqref{0217} leads to the following inequality:
\begin{align}
\int F_{s+1}\log^+ F_{s+1}d\nu_{s+1}\leq \int F_s\log^+ F_sd\nu_s+\int F_{s+1}\log^+\frac{F_{s+1}}{F_s}d\nu_{s+1}.
\end{align}

Since $F_{s+1}=\mathbb E_{s+1}(F_0)+\mathbb E_{s+1}(F')$, we have 
\begin{align}\nonumber
&\int F_{s+1}\log^+\frac{F_{s+1}}{F_s}d\nu_{s+1}\\\nonumber =&\int_{\mathbb E_{s+1}(F_0)\geq \mathbb E_{s+1}(F') } F_{s+1}\log^+\frac{F_{s+1}}{F_s}d\nu_{s+1} +\int_{\mathbb E_{s+1}(F_0)< \mathbb E_{s+1}(F') } F_{s+1}\log^+\frac{F_{s+1}}{F_s}d\nu_{s+1} \label{1532}
\\\leq& \int F_{s+1}\log^+\left(\frac{2\mathbb E_{s+1}[F_0]}{F_s}\right)d\nu_{s+1}\\&+2\int \mathbb E_{s+1}[F']\log^+\frac{F_{s+1}}{F_s} d\nu_{s+1}\label{1533}
\end{align}

We have 
\begin{align}\nonumber
\eqref{1532}\leq &\left( \int F_{s+1} \right)\log 2d\nu_{s+1}+ \int F_{s+1}\log^+\left(\frac{\max_{\theta\in \mathcal O/\mathcal Q_{s+1}}\mathbb E_{s+1}[F_0](\xi, \theta) }{F_s(\xi)}\right)d\nu_{s+1}\\
\leq & \log 2 +\int \max_{\theta\in \cO/\mathcal Q_{s+1}}\mathbb E_{s+1}[F_0](\xi, \theta) d\nu_s(\xi) <2+\log 2,
\label{1545}
\end{align}
where in the last step we have used \eqref{2108} and that $\log^+ x< x $ if $x>1$ .  

Since $\frac{F_{s+1}}{F_s}\leq q_{s+1}$, it follows that 
\begin{align}\label{1544}
\eqref{1533}<(2\log q_{s+1})\cdot \int\mathbb E_{s+1}[F']< 2r_1r_2\rho_1\log q_{s+1}.
\end{align}

Therefore, from \eqref{1545}, \eqref{1544}, we have
\begin{align}
\int F_{s+1}\log^+F_{s+1}d\nu_{s+1} < \int F_s\log^+ F_sd\nu_s+(2+\log 2)+2r_1r_2\rho_1 \log q_{s+1}.
\end{align}
Write $\tilde q_T=q_1\cdots q_T$ and sum in $s\leq T$, and we obtain 
\begin{align} \label{11321}
\int F_T\log^+ F_T\hspace{1mm}d\nu<(2+\log 2)T+2r_1r_2\rho_1\sum_{1\leq j\leq T}\log q_j <\log (10^{T} (\tilde q_T)^{2r_1r_2\rho_1}).
\end{align}

Therefore, if we let 
$$X=|\pi_{\cQ_1\cdots\cQ_T}(\text{supp}(\underbrace{(\mu_1\otimes\cdots \otimes \mu_{r_1})*\cdots * (\mu_1\otimes\cdots\otimes \mu_{r_1})}_{r_2 \text{ copies}}))|,$$
then it follows from \eqref{11321} that 

\begin{align}\nonumber
1=&\int F_Td\nu_T \leq (10^{T}(\tilde q_T)^{2r_1r_2\rho_1})^2\nu_T[\text{supp} (F_T)]+\int_{F_T>(10^{T}(\tilde q_T)^{2r_1r_2\rho_1})^2} F_Td\nu_T \nonumber\\
<&(10^T(\tilde q_T)^{2r_1r_2\rho_1})^2(\tilde q_T)^{-1}X\cdot |Z|+\frac{1}{2\log (10^T (\tilde q_T)^{2r_1r_2\rho_1})}\left(\int F_T\log^+ F_T\right) \nonumber\\
< &10^{2T}(\tilde q_T)^{4r_1r_2\rho_1-1}X\cdot |Z|+\frac{1}{2},
\end{align}
where
$$|Z|=\prod\bar{q}_j\leq q^{\rho_4}<(\tilde q_T)^{\frac{\kappa}{4}}.$$

Since $10^T< q^{0+}$, it then follows that 
$$X>(\tilde q_T)  ^{1-\kappa}.$$
An application of Corollary \ref{2346} finishes the proof of Theorem \ref{sp1}. 

\newpage
\section{Decompositions of the set $\cO/\fa$ \label{decomposition}}

Recall $\fa=\prod_{i\in I} \mathcal P_i^{n_i}=\prod_i\mathcal Q_i, q=|\mathcal O/\fa|$ and $q_i=|\mathcal O/\mathcal Q_i|$.  The goal of this section is to construct $\hat{\mathcal Q}=\mathcal Q_{j_1}\mathcal Q_{j_2}\cdots$, where ${j_1}, j_2$ is a sequence so that the decompositions of $\cO/\fa$ at \eqref{1339} exists with \eqref{1706} and \eqref{1621} satisfied, and also 
\begin{align}\label{1620}
|\mathcal O/\hat{\mathcal Q}|>|\cO/\fa|^{\rho_0}.
\end{align}

Assume we have already specified $j_1,\cdots, j_s$. Rearranging indices, we assume $j_1=1, \cdots, j_s=s$.  Set $\tilde{Q}=\mathcal Q_{1}\cdots\mathcal Q_{s}$,  and $\tilde{q}=q_{1}\cdots q_{s}$.  Assume 
\begin{align}\label{1323}
\tilde{q}<q^{\rho_0}.
\end{align}
We need to choose $j\neq 1,\cdots, s$ such that for each $1\leq i\leq r_1$, the set $\cO/\fa$ admits a decomposition $\cO/\fa=B_i^{j}\cup C_i^{j}$ that  

\begin{align}
\label{1601}\mu_i(C_i^{j})<\rho_1,
\end{align}
and for any $\xi\in \cO/\prod_{1\leq i\leq s}\cQ_i$, for any congruential ring $R<\cO/\cQ_{s+1}$ and $a,b\in \cO/\cQ_{s+1}$ with 
\begin{align}\label{1619}
[\cO/\cQ_{s+1}:bR]>q_{s+1}^{\rho_2},
\end{align}
the measure $\mu_i$ satisfies the non-concentration condition: 
\begin{align}\label{1620}
\nonumber&\mu_i\left(\{x\in B_i^{(s)}: \pi_{\cQ_1\cdots \cQ_s}(x)=\xi, \pi_{\cQ_{j}}\in a+bR\}\right)\\<&[\cO/\cQ_{j}: bR]^{-\rho_3}\mu_i\left(\{x\in B_i^{(s)}: \pi_{\cQ_1\cdots \cQ_s}(x)=\xi\}\right).
\end{align}

Fix $j\in I-\{1, 2, \cdots, s\}$.  We first introduce a set of dyadic indices $\mathfrak A_j=\{\alpha_k\}\subset [0, n_j]$ of size $\sim \log \frac{1}{\rho_2}$, where $\alpha_0=0$, $ \alpha_1=[\frac{\rho_2 n_j}{2}],  \alpha_2=[{\rho_2 n_j}],  \alpha_3=[{4\rho_2 n_j}], \cdots.$ 

For $t_1, t_2\in \mathfrak A_j$ with $t_1< t_2$, and for $K'<K$ with $\cO'$ the ring of integers of $K'$, define the test set:  $$\frak S_j(\cO'; t_1, t_2)= \{a+bR\subset \cO/\cQ_j: \frak p_j^{t_1}|| b,  R= R_{\cO', t_2-t_1} \},$$
recalling the definition of $R_{\cO', t_2-t_1}$ at \eqref{1329}.

The total number of such $\frak S(\cO'; t_1, t_2)$ is bounded by $C(d)[\log \frac{1}{\rho_2}]^2$. \par

To verify condition \eqref{1620}, instead of screening all $a+bR$, it suffices to consider $a+bR \in \frak S(\cO'; t_1, t_2)$ for all $t_1, t_2\in \mathfrak A_j$ and for all $\cO'<\cO$.  If all members in $$\cup_{\cO'<\cO}\cup_{t_1<t_2}\frak S(\cO'; t_1, t_2)$$ satisfy \eqref{1620}, then all $a+bR$ satisfing \eqref{1619} will also satisfy \eqref{1620}, with $\rho_3$ replaced by $\frac{\rho_3}{4}$.   

\medskip

Let $\xi\in\cO/\tilde{\mathcal Q}$ and $W\subset \cO/\cQ_j$.  
Define the tubes
$$S(\xi)=\{x\in \cO/\fa: \pi_{\tilde{Q}}(x)=\xi\},$$
and 
$$S_j(\xi,W)=\{x\in S(\xi): \pi_{\cQ_j}(x)\in W \}.$$

Fix $\mu_i$, fix $\frak S_j(\cO'; t_1,t_2)$, and we also prefix an order in $\frak S_j(\cO'; t_1, t_2)$. \par 
For any $a+bR\in \frak S_j(\cO'; t_1,t_2)$, let
$$q_{j,\cO'; t_1,t_2}=[\cO/\cQ_j: bR]={|\cO/\cP_j|^{t_1}|\cO'/(\mod \cP_j)|^{t_2-t_1}}$$
  Note that the definition for $q_j(\cO'; t_1,t_2)$ is independent of the choice of $a+bR$.  

Fix $\xi\in \cO/\tilde{\cQ}$. We define the set $C_{i,j}^{\cO'; t_1, t_2}(\xi)$ recursively as follows:  We initially set $C_{i,j}^{\cO'; t_1, t_2}(\xi)=\emptyset$.  Suppose $C_{i,j}^{\cO'; t_1, t_2}(\xi)$ has been defined after screening all elements before $a+bR$ in $\frak S_j(\cO'; t_1,t_2)$.  If 
\begin{align}\label{0947}
\mu_i(S(\xi, a+bR)- C_{i,j}^{\cO'; t_1,t_2}(\xi))\geq q_{j,\cO'; t_1,t_2}^{-\rho_3}(\mu_i S(\xi))
\end{align}
then throw the whole set $S(\xi, a+bR)$ into $C_{i,j}^{\cO'; t_1,t_2}(\xi)$.  The set
$C_{i,j}^{\cO'; t_1, t_2}(\xi)$ is then determined after screening all elements in $\frak S_j(\cO'; t_1,t_2)$.   It is important to note that along this screening process, at most $q_{j,\cO'; t_1,t_2}^{\rho_3}$ many $a+bR$ {are picked up}.  


If $$\mu_i(S(\xi)-\cup_{\cO'<\cO}\cup_{t_1<t_2} C_{i,j}^{\cO'; t_1, t_2}(\xi))\geq \frac{\rho_1}{2}\mu_i(S(\xi)),$$ define
$$B_i^{j}(\xi)=S(\xi)- \cup_{\cO'<\cO}\cup_{t_1<t_2} C_{i,j}^{\cO'; t_1, t_2}(\xi)$$
and 
$$\tilde{C}_i^{j}(\xi)=\emptyset.$$
Otherwise,  define
$$B_i^{j}(\xi)=\emptyset, $$
and
$$\tilde{C}_i^{j}(\xi)=S(\xi)- \cup_{\cO'<\cO}\cup_{t_1<t_2} C_{i,j}^{\cO'; t_1, t_2}(\xi).$$
Then we let 
\begin{align*}
&B_i^{j}=\cup_{\xi\in\cO/\tilde{Q}}B_i^{j}(\xi),\\
&C_{i,j}^{\cO'; t_1, t_2}=\cup_{\xi\in \cO/\cQ_j}C_{i,j}^{\cO'; t_1, t_2}(\xi),\\
&\tilde{C}_i^{j}=\cup_{\xi\in\cO/\tilde{Q}}\tilde{C}_i^j(\xi).
\end{align*}

Thus we arrived at a decomposition 

$$\cO/\fa=B_i^{j}\sqcup C_i^{j} $$
with
$$C_{i}^{j} =\left(\cup_{\cO'<\cO}\cup_{t_1<t_2}C_{i,j}^{\cO'; t_1, t_2}\right)\sqcup \tilde C_i^{j},$$

$$\mu_i(\tilde{C}_i^{(j)})<\frac{\rho_1}{2},$$
and $B_i^{j}$ satisfies condition \eqref{1620} for $B_i^{(s)}$ if $q_j$ is sufficiently large.

%
\medskip

If given $j$, there is $1\leq i_j\leq r_1$ that fails \eqref{1601}, then for some $\cO_j<\cO$ and for some $t_{j,1}, t_{j,2}\in \mathfrak A_j$ with $t_{j,1}<t_{j,2}$, we have

\begin{align}
\mu_{i_j}(C_{i_j,j}^{\cO_j; t_{j,1},t_{j,2}})\geq \frac{\rho_1}{2C(d)[\log\frac{1}{\rho_2}]^2}. 
\end{align}

Furthermore, assume for all $j\not\in\{1,\cdots, s\}$, there is $1\leq i\leq r_1$ that fails \eqref{1601}.  Then by a standard pigeon hole argument, there exists an index $1\leq i_0\leq r_1$, a family of indices $I'\subset I-\{1,\cdots, s\}$, such that $i_j=i_0$, 
and 
\begin{align}
\prod_{j\in I'}q_j\geq \left(\prod_{j\in I-\{1,\cdots, s\}}q_j\right)^{\frac{1}{r_1}}=(\frac{q}{\tilde q})^{\frac{1}{r_1}}.
\end{align}

Write $$q_j^*=q_{j, \cO_j; t_{j,1},t_{j,2}}$$
and 
$$C_j^{*}=C_{i_0,j}^{\cO_j; t_{j,1},t_{j,2}}.$$

Then 
\begin{align}\label{1222}
\sum_{j\in I'}\log q_{j}^* \cdot \mu_{i_0}(C_j^*)\geq \frac{\rho_1}{2C(d)\log [\frac{1}{\rho_2}]^2}\log q', 
\end{align}

where 

\begin{align}\label{0846}
q'=\prod_{j\in I'}q_{j}^*\geq \prod_{j\in I'}q_j^{\rho_2}>  (\frac{q}{\tilde q})^{\frac{\rho_2}{r_1}}.
\end{align}

If we write $z_j=\frac{\log q_{j}^*}{\log q'}$, then $\sum_j z_j=1$ and \eqref{1222} is equivalent to 

\begin{align} \label{0847}
\sum_{j\in I'}z_j\mu_{i_0}(C_j^*)\geq \frac{\rho_1}{2C(d)[\log\frac{1}{\rho_2}]^2},
\end{align}

For a set $J\subset I'$, write $$z_J=\sum_{j\in J}z_j$$ and $$C_J=\cap_{j\in J}C_j^*-\cup_{j\in I'-J}C_j^*.$$

Then \eqref{0847} can be rewritten as 
\begin{align} \label{1356}
\sum_{J\subset I'} z_J\mu_{i_0}(C_J) \geq \frac{\rho_1}{2C(d)[\log\frac{1}{\rho_2}]^2}.
\end{align}

Since $$\sum_J \mu_{i_0}(C_J)=\mu_{i_0}(\cup_{j\in I'} C_j^*)\leq 1,$$ we have from \eqref{1356}, 
\begin{align}\label{1451}
\sum_{J\subset  I'} \mu_{i_0}(C_J){\boldsymbol 1}\left\{z_J>\frac{\rho_1}{4C(d)[\log\frac{1}{\rho_2}]^2}\right\}>\frac{\rho_1}{4C(d)[\log\frac{1}{\rho_2}]^2}.
\end{align}

The number of summands on the left hand side of \eqref{1451} is bounded by $2^{|I|}$.  Therefore, there exists $J\subset  I-\{1,\cdots, s\}$, such that 
\begin{align}\label{2049}
z_J\geq \frac{\rho_1}{4C(d)[\log\frac{1}{\rho_2}]^2},
\end{align}
and 
\begin{align}\label{2256}
\mu_{i_0}(C_J)\geq \frac{\rho_1}{2^{|I|+2}C(d)[\log\frac{1}{\rho_2}]^2}.
\end{align}

{
Since $C_j$ lies in the intersection of $C_j^{*}=\cup_{\xi\in \cO/\tilde{\cQ}}C_j^*(\xi)$ and each $C_j^*(\xi)$ lies in the union of at most $q_j^{\rho_3}$ affine translates, by a standard probabilistic argument, there exists $\xi_0\in \cO/\tilde{\cQ}$, $S(\xi, a_j+b_jR_j) \subset C_j^{*}(\xi_0), \forall j\in J$, such that  
$$\mu_{i_0}(\cap_{j\in J}S(\xi, a_j+b_jR_j )>(\tilde{q})^{-1}(\prod_{j\in J}q_j^*)^{-\rho_3}\frac{\rho_1}{2^{|I|+2}C(d)[\log\frac{1}{\rho_2}]^2}$$
recalling $q_j^*=[\cO/\cQ_j: b_jR_j]$.  \par

If we let the affine translate $a+bR\subset \cO/\fa$ be the intersection of $S(\xi_0)$ and the preimage of $\prod_{j\in J} a_j+b_jR_j$ under the canonical map: $\cO/\fa \mapsto \prod_{j\in J}\cO/\cQ_j$, then 
}
 
\begin{align}\label{1454}
\nonumber& \pi_{\tilde Q}(a+bR)=\xi_0, \\
\nonumber& \pi_{\cQ_j}(a+bR)= a_j+b_jR_j \text{ for }j\in J,\\
\nonumber& \pi_{\cQ_j}(a+bR)=\cO/\cQ_j \text{ for }j\not\in [1,\cdots, s]\cup J, \\
\nonumber& [\cO/\fa: bR]=\tilde{q}\prod_{j\in J}q_j^*,\\
& \mu_{i_0}(a+bR)>(\tilde{q})^{-1}(\prod_{j\in J}q_j^*)^{-\rho_3}\frac{\rho_1}{2^{|I|+2}C(d)[\log\frac{1}{\rho_2}]^2}.
\end{align}
From \eqref{2049}, we have 
\begin{align}\label{2105}
\prod_{j\in J}q_j^*>(\frac{q}{\tilde q})^{\frac{\rho_1\rho_2}{4C(d)r_1[\log\frac{1}{\rho_2}]^2}}.
\end{align} 

On the other hand, 
\begin{align}
[\cO/\fa: bR]=\tilde{q}\prod_{j\in J}q_j^*> q^{\frac{\rho_1\rho_2}{4C(d)r_1[\log \frac{1}{\rho_2}]^2}}\geq q^{\epsilon},
\end{align}
by our choice of $\epsilon$ at \eqref{1748}. 
By the non-concentration property \eqref{nonconcentration} of $\mu_{i_0}$, we have

\begin{align}\label{1455}
\mu_{i_0}(a+bR)< (\tilde{q}\prod_{j\in J}q_j^*)^{-\gamma}.
\end{align}

Recalling $\tilde{q}<q^{\rho_0}$ and \eqref{2105}, our choices of parameters at \eqref{1748} make \eqref{1454} and \eqref{1455} contradict when $q$ is sufficiently large.

\newpage
\appendix
\section{Some Elementary Lemmas}

Let $\chi$ be a primitive additive character on $\cO/\prod_{i\in I}\cP_i^{n_i}$.  Then, all additive characters on $\cO/\prod_{i\in I}\cP_i^{n_i}$ are given by 
$$\chi_z(x):=\chi(zx), \hspace{2mm} z\in \cO/\prod_{i\in I}\cP_i^{n_i}.$$

\begin{lem} \label{0715} Let $\gamma_1,\gamma_2>0, \gamma_1+\gamma_2>1$ and $k\in \mathbb Z_+$ be such that $k>\frac{4}{\gamma_1+\gamma_2-1}$. 
Let $A_j,B_j\subset \mathcal O/\prod_{i\in I} \mathcal P_i^{n_i}(1\leq i\leq k)$ satisfying, for any $\prod_{i\in I}\mathcal P_i^{m_i}, m_i\leq n_i$, we have 
\begin{align}\nonumber\max_{\xi\in \mathcal O/ \prod_{i\in I}\mathcal P_i^{m_i}}|x\in A_j | \pi_{\prod_{i\in I}\mathcal P_i^{m_i}} (x)=\xi |<  |\cO/\prod_i \cP_i^{m_i} |^{-\gamma_1}|A_j|. \\
\max_{\xi\in \mathcal O/ \prod_{i\in I}\mathcal P_i^{m_i}}|x\in B_j | \pi_{\prod_{i\in I}\mathcal P_i^{m_i}} (x)=\xi |<  |\cO/\prod_i \cP_i^{m_i} |^{-\gamma_2}|B_j|.\label{0712}
\end{align}

Let $\nu$ be the image measure on $R$ of the normalized counting measure on $\prod_{j=1}^k(A_j\times B_j)$ under the map

$$(x_1,y_1,\cdots, x_k, y_k)\rightarrow x_1y_1+\cdots+x_ky_k.$$
Then for any $z\in \mathcal O/\prod_{i}\cP_i^{n_i}$ with $\langle \pi_{\cO/\prod_{i\in I}\cP_i^{n_i}}^{-1}(z), \prod_{i\in I} \mathcal P_i^{n_i}\rangle=\prod_{i\in I} \mathcal P_i^{m_i},$ we have 
$$|\hat{\nu}(\chi_z)|< q_z^{-\frac{k(\gamma_1+\gamma_2-1)}{2}},$$
where  \begin{align*}\label{1127}
q_z:=\prod_{i\in I} |\cO/\cP_i|^{(n_i-m_i)}. 
 \end{align*}
 As a result, if $k(\gamma_1+\gamma_2-1)\geq 4$,  we have $$\sum_{j=1}^{k}A_jB_j\supset \mathcal O/\prod_i\mathcal P_i^{n_i}. $$
\end{lem}
\begin{proof}
For $\xi\in \mathcal O/\prod_i \mathcal P_i^{n_i}$, we have 
\begin{align}
\nu(\xi)=& \frac{1}{\prod_{j=1}^k |A_j||B_j|}|\{(x_1,y_1,\cdots, x_k,y_k)\in A_1\times B_1\times\cdots \times A_k\times B_k| x_1y_1+\cdots+ x_ky_k=\xi\}| \\ 
\label{1122}= &\frac{1}{|\cO/\prod_i \cP_i^{n_i}| } \frac{1}{\prod_{j=1}^k |A_j||B_j|}\sum_{\substack{z\in\mathcal O/\prod_i\mathcal P_i^{n_i} }}\sum_{\substack {x_i\in A_i, y_i\in B_i}}\chi_z (\xi-x_1y_1-\cdots-x_ky_k)
\end{align}

We estimate $|\sum_{x_j\in A_j, y_j\in B_j}\chi_z (\xi-x_1y_1-\cdots-x_ky_k)|$.  Suppose 
$$z\cdot \cO/\prod_{i\in I}\cP_i^{n_i}=\prod_{i\in I}\cP_i^{m_i}/\prod_{i\in I}\cP_i^{n_i}$$
For $\xi\in \mathcal O/\prod_i\mathcal P_i^{n_i-m_i}$, we set 
\begin{align}
\eta_1^{(j)}(\xi)=| \{x\in A_j| \pi_{\prod_i\mathcal P_i^{n_i-m_i}}(x)=\xi \} |< q_z^{-\gamma_1}|A_j|
\end{align}
and 
\begin{align}
\eta_2^{(j)}(\xi)=| \{x\in B_j| \pi_{\prod_i\mathcal P_i^{n_i-m_i}}(x)=\xi \} |< q_z^{-\gamma_2}|B_j|,
\end{align}
where \begin{align*}
q_z:=\prod_{i\in I} |\cO/\cP_i|^{(n_i-m_i)}. 
 \end{align*}
Then 
\begin{align*}
&|\sum_{x_i\in A_i, y_i\in B_i}\chi_z (\xi-x_1y_1-\cdots-x_ky_k)|\\
=&\prod_{j=1}^k | \sum_{\xi_1,\xi_2\in \mathcal O/\prod_i\mathcal P_i^{m_i}}\eta_1^{(j)}(\xi_1)\eta_2^{(j)}(\xi_2)\chi_z(\xi_1\xi_2)| \\
\leq &\prod_{j=1}^k\left(\left[\sum_{\xi_1\in  \mathcal O/\prod_i\mathcal P_i^{n_i-m_i}} \eta_1^{(j)}(\xi_1)^2 \right]^{1/2}\left[ \sum_{\xi_1\in  \mathcal O/\prod_i\mathcal P_i^{n_i-m_i}} |\sum_{\xi_2\in  \mathcal O/\prod_i\mathcal P_i^{n_i-m_i}}\eta_2^{(j)}(\xi_2)\chi_z(\xi_1\xi_2)|^2  \right]^{1/2}\right) \\
=&\prod_{j=1}^k\left(\left[\sum_{\xi_1\in  \mathcal O/\prod_i\mathcal P_i^{n_i-m_i}} \eta_1^{(j)}(\xi_1)^2 \right]^{1/2}\left[ q_z \sum_{\xi_2\in  \mathcal O/\prod_i\mathcal P_i^{n_i-m_i}}\eta_2^{(j)}(\xi_2)^2   \right]^{1/2}\right) \\
\leq & q_z^{\frac{k(1-\gamma_1-\gamma_2)}{2}}\prod_{j=1}^k |A_j||B_j|
\end{align*}
Therefore, 
$$\frac{1}{|\cO/\prod_i \cP_i^{n_i}| } \frac{1}{\prod_{j=1}^k |A_j||B_j|}|\sum_{\substack{z\in\mathcal O/\prod_i\mathcal P_i^{n_i}-\{0\} }}\sum_{\substack {x_i\in A_i, y_i\in B_i}}\chi_z (\xi-x_1y_1-\cdots-x_ky_k)|$$ is majorized by 
\begin{align*}&\frac{1}{|\cO/\prod_i \cP_i^{n_i}| }\sum_{z\in \mathcal O/\prod_{i}(p_i^{n_i})}q_z^{-\frac{k(\gamma_1+\gamma_2-1)}{2}} \\ \leq & \frac{1}{|\cO/\prod_i \cP_i^{n_i}| }\sum_{q\geq 2} q^{-\frac{k(\gamma_1+\gamma_2-1)}{2}}\\\leq &\frac{1}{|\cO/\prod_i \cP_i^{n_i}| }\left(\frac{1}{2^{\frac{k(\gamma_1+\gamma_2-1)}{2}}}+\int_{2}^\infty \frac{1}{x^{\frac{k(\gamma_1+\gamma_2-1)}{2}}}dx\right) \\ \leq&\frac{1}{|\cO/\prod_i \cP_i^{n_i}| }\left(\frac{1}{2^{\frac{k(\gamma_1+\gamma_2-1)}{2}}}+ \frac{2^{-\frac{k(\gamma_1+\gamma_2-1)}{2}+1}}{\frac{k(\gamma_1+\gamma_2-1)}{2}}\right)\\ \leq &\frac{1}{2|\cO/\prod_i \cP_i^{n_i}| },
\end{align*}
if $k(\gamma_1+\gamma_2-1)\geq 4$.  \par
Therefore, 
$$\nu(\xi)\geq \frac{1}{4|\cO/\prod_i \cP_i^{n_i}| },$$
and the lemma is thus proved. 
\end{proof}

\begin{cor} \label{2346} Let $A\subset \mathcal O/ \prod_i \mathcal P_i^{n_i}$ with $|A|>q^{1-\gamma}$, $\gamma<\frac{1}{10}$, then there exists $\prod_{i}\cP_i^{l_i}$, $l_i\leq n_i$, such that  $$|\cO/\prod_{i\in I}\cP_i^{l_i}|<|\cO/\prod_i \cP_i^{n_i}|^{\frac{6\gamma}{5}},$$ and $$ \prod_{i\in I}\cP_i^{2l_i}/ \prod_i \mathcal P_i^{n_i}\subset \sum_{24}A^2-\sum_{24}A^2.$$
\end{cor}
\begin{proof}  First we find a minimal translate $a+\prod_{i\in I}\cP_i^{l_i}$
 such that 
$$|A\cap (a+\prod_{i\in I}\cP_i^{l_i}/\prod_{i\in I}\cP_i^{n_i})|\geq  |\cO /\prod_{i\in I}\cP_i^{l_i}|^{-\frac{2}{3}}|A|.$$
and we let $A'\subset \prod_i\mathcal P^{n_i-l_i}$ be such that $x\mapsto a+(\prod_{i\in I}\fp^{l_i})x$ gives a bijection between $A'$ and $A\cap  (a+\prod_{i\in I}\cP_i^{l_i}/\prod_{i\in I}\cP_i^{n_i})$. 

Then we have
\begin{align}
&|\cO/\prod_i \cP_i^{l_i}|< |\cO/\prod_i \cP_i^{n_i}|^{\frac{3\gamma}{5}},\\
&|A'|>|\cO/\prod_i \cP_i^{n_i}|^{1-\frac{7\gamma}{5}}.
\end{align}
If we let $A_i=B_i=A'$ in the Lemma \ref{0715} with modulus replaced by $\prod_i \mathcal P_i^{n_i-l_i}$, then the assumption for $A_i, B_i$ is satisfied with $\gamma_1=\gamma_2=\frac{2}{3}$.  An application of Lemma \ref{0715} with $k=12$ leads to 
$$\sum_{12}(A')^2=\mathcal O/\prod_i\mathcal P_i^{n_i-l_i},$$
which implies 
$$\sum_{24}A^2-\sum_{24}A^2\supset  (\prod_i \cP_i^{2l_i})/ \prod_i\mathcal P_i^{n_i}.$$
\end{proof}

Next, We invoke a lemma about symmetric polynomials:
\begin{lemma}
\label{polyiden}
Fix $k\in \mathbb Z_+$. Define
\[P_t=\sum\limits_{1\leq i_1<i_2<\cdots <i_t\leq k} \left(\sum\limits_{j=1}^tx_{i_j}\right)^k\in \mathbb{Q}[x_1, x_2, \ldots, x_k]\]
for each $1\leq t\leq k$. Then 
\[\sum\limits_{i=0}^{k-1}(-1)^iP_{k-i}=k!\ \prod_{i=1}^kx_i.\]
\textbf{Remark.} The identities for $k=2, 3$ are
\[(a+b)^2-(a^2+b^2)=2ab\]
and
\[(a+b+c)^3-\left((a+b)^3+(a+c)^3+(b+c)^3\right)+(a^3+b^3+c^3)=6abc.\]
\end{lemma}
\textit{Proof.} It suffices to check that the coefficient for all degree-$k$ monomials in $x_1, x_2, \ldots, x_k$ are the same on both sides. Consider the monomial $f=x_{i_1}^{\alpha_1}x_{i_2}^{\alpha_2}\cdots x_{i_j}^{\alpha_j}$ where $1\leq j \leq k$, $1\leq i_1<i_2<\cdots <i_j\leq k$ and $\sum\limits _{l=1}^ja_{i_l}=k.$ The coefficient is non-zero in $P_t$ if and only if $j\leq t\leq k$, and in these cases the coefficient of $f$ in $P_t$ is
\[\binom{k-j}{t-j}\frac{k!}{\alpha_1!\alpha_2!\cdots \alpha_j!}.\]
If $j=k$, then $f=x_1x_2\cdots x_k$ is the only possible case and the coefficients of $f$ on both sides are both $k!$. Therefore, it suffices to show
\[\sum\limits_{t=j}^k(-1)^{k-t}\binom{k-j}{t-j}\frac{k!}{\alpha_1!\alpha_2!\cdots \alpha_j!}=0\]
for all $j\leq k-1, j\leq t\leq k-1$, which is equivalent to
\[\sum\limits_{t=j}^k(-1)^{k-t}\binom{k-j}{t-j}=0\Leftrightarrow \sum\limits_{t=0}^{k-j}(-1)^{k-j-t}\binom{k-j}{t}=0.\]
It's clear that the last identity follows from the binomial expansion
\[0=(1-1)^{k-j}=\sum\limits_{t=0}^{k-j}(-1)^{k-j-t}\binom{k-j}{t}.\]
This gives the proof of Lemma $\ref{polyiden}$.
\\

\newpage
\bibliographystyle{alpha}

\bibliography{boundedgenerationV3}

\end{document}